\documentclass[11pt,english,twoside,a4paper]{amsart}     
\usepackage[english,frenchb]{babel}
\usepackage{amssymb,amsmath,amsthm,stmaryrd}
\usepackage{mathrsfs}
\usepackage{graphicx,psfrag,epsfig}
\usepackage{graphicx}
\usepackage[T1]{fontenc}
\usepackage[latin1]{inputenc}
\usepackage[all]{xy}
\usepackage{mathrsfs}
\usepackage{pst-all}
\usepackage{float}
\theoremstyle{plain}

\theoremstyle{plain}

\newtheorem{thm}{Theorem}
\newtheorem{prop}{Proposition}[section]
\newtheorem{lem}{Lemma}[section]

\newtheorem{nota}{Notation}[section]

\newtheorem{cor}{Corollary}[section]

\theoremstyle{definition}
\newtheorem{Def}{Definition}[section]
\newtheorem{exm}{Example}[section]
\theoremstyle{remark}
\newtheorem{rem}{Remark}[section]
\newtheoremstyle{restate}
{\topsep}
{\topsep}
{\itshape}
{}
{\bfseries}
{.}
{ }
{\thmname{#1}\thmnote{ #3}}
\theoremstyle{restate}

\newtheorem{thm*}{Theorem}
\newtheorem{prop*}{Proposition}
\newtheorem{lem*}{Lemma}
\newtheorem{add*}{Addendum}
\newtheorem{cor*}{Corollary}
\newtheorem{pte*}{Property}
\theoremstyle{definition}
\newtheorem{Def*}{Definition}
\newtheorem{exm*}{Example}
\theoremstyle{remark}
\newtheorem{rem*}{Remark}

\theoremstyle{definition}

\theoremstyle{remark}

\newcommand{\ov}{\overline}
\newcommand{\argsh}{\mathop{\rm argsh\,}\limits}

\newcommand{\Inf}{\mathop{\rm Inf\,}\limits}
\newcommand{\Max}{\mathop{\rm Max\,}\limits}
\newcommand{\Min}{\mathop{\rm Min\,}\limits}

\newcommand{\diam}{{\rm diam\,}}

\newcommand{\ha}{\widehat}
\newcommand{\ti}{\tilde}
\newcommand{\wti}{\widetilde}
\newcommand{\N}{\mathbb N}
\newcommand{\Z}{\mathbb Z}

\newcommand{\R}{\mathbb R}

\newcommand{\A}{\mathbb A}
\newcommand{\T}{\mathbb T}

\newcommand{\Cal}{\mathcal}

\newcommand{\mcr}{\mathscr}

\newcommand{\demi}{\frac{1}{2}}

\newcommand{\Log}{{\rm Log\,}}

\newcommand{\supp}{{\rm supp\,}}
\newcommand{\setm}{\setminus}
\DeclareMathOperator{\Vol}{Vol}
\DeclareMathOperator{\Card}{Card}

\newcommand{\rk}{{\rm rank\,}}

\newcommand{\Ss}{\Cal S}
\newcommand{\Tt}{\Cal T}
\newcommand{\Dd}{\Cal D}

\newcommand{\Cc}{\Cal C}
\newcommand{\Ee}{\Cal E}

\newcommand{\Zz}{\Cal Z}
\newcommand{\Pp}{\Cal P}
\newcommand{\Hh}{\Cal H}

\newcommand{\Ll}{\Cal L}
\newcommand{\Rr}{\Cal R}

\newcommand{\Gg}{\Cal G}

\newcommand{\CC}{\mcr C}
\newcommand{\TT}{\mcr T}
\newcommand{\EE}{\mcr E}

\newcommand{\PP}{\mcr P}

\newcommand{\MM}{\mcr M}
\newcommand{\VV}{\mcr V}

\newcommand{\tpi}{\ti{\pi}}

\newcommand{\tM}{\wti{M}}

\newcommand{\tDd}{\wti{\Dd}}

\newcommand{\al}{\alpha}
\newcommand{\be}{\beta}
\newcommand{\ga}{\gamma}
\newcommand{\sig}{\sigma}
\newcommand{\eps}{\varepsilon}
\newcommand{\Ga}{\Gamma}
\renewcommand{\th}{\theta}
\newcommand{\vp}{\varphi}

\newcommand{\Sig}{\Sigma}

\newcommand{\om}{\omega}
\newcommand{\Om}{\Omega}
\newcommand{\lam}{\lambda}

\newcommand{\de}{\delta}
\newcommand{\ups}{\upsilon}
\newcommand{\vpi}{\varpi}

\newcommand{\rit}{\rightarrow}
\newcommand{\ma}{\mapsto}

\newcommand{\ci}{\underline{c}}
\newcommand{\cs}{\overline{c}}

\newcommand{\inv}{^{-1}}
\newcommand{\lio}{_{e,\rho}}

\newcommand{\bvp}{\bar{\vp}}
\newcommand{\bs}{\bar{s}}
\newcommand{\bm}{\bar{m}}

\newcommand{\hto}{{\rm{h_{top}}}}
\newcommand{\hp}{{\rm{h_{pol}}}}

 \newcommand{\8}{\scalebox{1.3}{$\infty$}}

\begin{document}
\selectlanguage{english}

\author{Clémence Labrousse}
\title[Polynomial growth of volume of balls]{Polynomial growth of the volume of balls\\ for zero-entropy geodesic systems}
\thanks{Institut de Mathématiques de Jussieu, UMR 7586, {\em Analyse algébrique},
4, place Jussieu, 75005 Paris.
email: labrousse@math.jussieu.fr}

\date{}

\maketitle

\begin{abstract} 
The aim of this paper is to state and prove a polynomial analogue of the classical Manning inequality, relating
the topological entropy of a geodesic flow with the growth rate of the volume of balls in the universal covering. 
To this aim, we use a numerical conjugacy invariant for dynamical systems, the {\em polynomial entropy}. It is infinite when the topological entropy is positive.
We first prove that the  growth rate of the volume of balls is bounded above by means of the polynomial entropy of the geodesic flow. For the flat torus this inequality becomes an equality.
We then study  explicitely the example of the torus of revolution (which is a case of strict inequality). We  give an exact asymptotic equivalent of the growth rate of volume of balls.
\end{abstract}


\section{Introduction and main results}

Let $(M,g)$ be a compact Riemannian manifold. The \emph{Hamiltonian geodesic flow} is the flow associated with the geodesic Hamiltonian $H$ on $T^*M$ defined by
\[
\begin{matrix}
H\;\; : & T^*M & \longrightarrow &  \R\\
  & (m,p) & \mapsto & g_m^*(p,p)
 \end{matrix}
\]
where $g_m^*$ is deduced from $g_m$ by the Legendre transform. 
The projections on $M$ of the solutions of the flow are the geodesics of the metric. Throughout this paper, we consider only smooth metrics with complete geodesic flows.

\subsection{Topological and polynomial entropies.} Given a compact metric space $(X,d)$ and a continuous flow $\phi:\R\times X\rit X: (t,x)\ma\phi_t(x)$, for each $t>0$, one can define the dynamical metric
\begin{equation}
d^{\phi}_t(x,y)=\max_{0\leq k\leq t}d(\phi_t(x),\phi_t(y)).
\end{equation}
All these metrics $d^{\phi}_t$ are equivalent and define the same topology as $d$ on $X$. In particular, $(X,d^{\phi}_t)$ is compact. So for any $\varepsilon>0$, $X$ can be covered by a finite number of balls of radius $\varepsilon$ for $d^{\phi}_t$. Let $G_t(\varepsilon)$ be the minimal number of balls of such a covering.
The topological entropy of $f$ is defined by
\[
\hto(\phi)=\lim_{\varepsilon\rightarrow 0}\limsup_{t\rightarrow\infty}\frac{1}{t}\log G_t(\varepsilon).
\]
When the topological entropy vanishes, the complexity of dynamical systems can be described by {\em several} natural non equivalent conjugacy invariants\,-\,the polynomial entropies\,-\, which depict the polynomial growth rate of the number $G_t(\varepsilon)$ (for a complete introduction see \cite{Mar-09}). 
In this paper we will focus on the {\em{strong polynomial entropy}} $\hp$ (or polynomial entropy, for short). The definition of $\hp$ is recalled in Section 2.

It has been proved by A. Manning (\cite{M-79}) that the topological entropy of the geodesic flow $\phi$ restricted to the unit tangent bundle $SM$ is related to the growth rate of volume on the universal cover $\wti{M}$ of $M$, endowed with the lifted Riemannian metric. More precisely, Manning showed that if $B(x,r)$ is the ball in $\wti{M}$ centered at $x$ and of radius $r$ and if $V(x)$ is defined by
\[
V(x)=\limsup_{r\rightarrow\infty}\frac{1}{r}\,\Log\Vol B(x,r),
\]
then $V(x)\leq \hto(\phi)$ for all $x\in\wti{M}$ ($V(x)$ is in fact independent of $x$).

In this paper we focus on the case where $\hto(\phi)$ vanishes.  One may then expect the growth rate of the volume to be polynomial. Let $\tau(x)$ be defined by
\[
\tau(x)=\limsup_{r\rit\infty}\frac{\log \Vol B(x,r) }{\log r }=\Inf\left\{s\geq 0\,|\,\limsup_{r\rit\infty}\frac{1}{r^s}\Vol B(x,r)=0\right\}.
\]
We will show in Section 3.1 that $\tau(x)$ is independent of $x$. Indeed, it is the  degree of growth of the fundamental group $\pi_1(M)$. We denote it by $\tau(M)$. Section $3$ is devoted to the proof of the following result.

\begin{thm}\label{Manning} 
Let $(M,g)$ be a  compact Riemannian manifold  and let $\phi$ be the restriction of the
geodesic flow to the unit tangent bundle $SM$. Then
\[
\tau(M)\leq \hp(\phi)+1.
\]
\end{thm}
As a consequence,  since $\tau(\T^\ell)=\ell$, the polynomial entropy of a geodesic flow on $\T^\ell$ is larger than $\ell-1$. 
For a flat metric on $\T^\ell$, this inequality becomes an equality.
More generally, \emph{completely integrable}  geodesic systems (see definition below)  on the torus do \emph{minimize} the complexity.

\subsection{Entropy and integrable Hamiltonian systems.}
A Hamiltonian function $H$ on a symplectic manifold $M$ of dimension $2\ell$ is \textit{integrable in the Liouville sense} when there exists a smooth map $F = (f_i )_{1\leq i\leq\ell}$ from $M$ to $\R^{\ell}$, which is a submersion over an open dense domain $O$ and whose components $f_i$ are first integrals in involution of the Hamiltonian vector field $X^H$. The fundamental example is that of \emph{completely integrable systems} (one also say in ``action-angle'' form) on the annulus  $\A^{\ell}=\T^{\ell}\times\R^{\ell}$, that is, Hamiltonian systems of the form $H(\theta,r)=h(r)$. 
For instance, flat metrics on the torus are in action-angle form.
The structure of these systems is well known: $M$ is foliated by Lagrangian tori $\T^\ell\times\{r_0\}$ on which the flow is linear with frequency $\om(r_0):=dh(r_0)$.

In the general case, the regular set of the moment map $F$ is a countable union of  ``action-angle domains'' where the system is symplectically conjugate to a system in action-angle form (this is the Arnol'd-Liouville Theorem). It is well-known that the topological entropy of a system in action-angle form vanishes, so the entropy of a Liouville integrable system is ``localized'' on the singular set of the map $F$. 

It has been proved by Paternain (\cite{Pat-dim4}) that if $M$ is  $4$-dimensional and if the Hamiltonian vector field possesses a first integral $f$ (independent of $H$) such that, on a regular level $\EE$ of $H$, the critical points of $f$ form submanifolds, then the topological entropy of $\phi^H$ in restriction to $\EE$ vanishes.

In Sections 4-7, we consider the particular class $\TT_\MM$ of tori of revolution defined as follows:
\begin{equation}\label{deftore}
\TT_\MM:=\{\Sig_{x,y}(\R^2)\,|\,(x,y)\in\PP_\MM^+\times\PP\}\subset\{\Sig_{x,y}(\R^2)\,|\,(x,y)\in\PP^+\times\PP\}
\end{equation}
where is $\PP$ the space of $1$--periodic smooth functions $x:\R\rit\R$, $\PP^+$ the subspace of $\PP$ of positive functions, $\PP_\MM^+$ the subset of Morse functions $x\in\PP^+$ such that any critical value is reached once and where $\Sig_{x,y}$ is defined by
\[
\begin{array}{llll}
\Sig_{x,y}\;\; : & \R^2 &\longrightarrow &  \R^3\\
  & (\vp,s) & \mapsto & (x(s)\cos 2\pi\varphi,\,x(s)\sin 2\pi\varphi,\,y(s)).
 \end{array}
\]

\begin{figure}[h]
\begin{center}
\begin{pspicture}(2cm,2cm) 
\psset{xunit=0.5cm,yunit=.5cm}
\rput(0,0.5){ 
\psline{->}(-3,-2)(-3, 4)
\psline{->}(-3,-2)(3, -2)
\psarc{<-}(-3.2,2.6){1}{-130}{-40}
{\psecurve(2.5,0)(1.2,0.5)(1.5,1.1)(0.7,2)(1.2,2.6)(0.6,3)(0,3.2)(-1.3,3)(-2, 2.5)(-1.5, 1.8)(-1.3,1.7)
(-1,1.3)(-1.2,0.9)(-1.7,0.7)(-2.5,0)(-2,-0.7)(-1.5,-1)(0,-0.8)(1.3,-1)(2,-0.5)(2.5,0)(1.2,0.5)(1.5,1.1)(0.7,2)}
} 
\rput(3.3,-2.1){x}
\rput(-3.4,4.3){y}
\end{pspicture}
\vskip1cm
\caption{A meridian curve of a torus in $\TT_\MM$}
\end{center}
\end{figure}

Let $(x,y)\in\PP^+\times\PP$. For a sake of lightness, we just denote by $\Sig$ the map $\Sig_{x,y}$. We set $\Tt:=\Sig(\R^2)$, it is a compact surface homeomorphic to the torus $\T^2=\R^2\setm\Z^2$. We denote by $\ha{\Sig}$ the map defined on $\T^2$ such that the following diagramm commutes:
\[
\xymatrix{
    \R^2 \ar[r]^{\Sig}\ar[d]_\vpi  & \TT  \\
    \T^2\ar[ru]_{\ha{\Sig}}
  .}
\]
Hence
\[
\TT:=\{\ha{\Sig}_{x,y}(\T^2)\,|\,(x,y)\in\PP^+\times\PP\},\quad \TT_\MM:=\{\ha{\Sig}_{x,y}(\T^2)\,|\,(x,y)\in\PP^+_\MM\times\PP\}.
\]
\vspace{0.1cm}

When $\Tt\in\TT_\MM$ is endowed with the induced Euclidean metric, the geodesic Hamiltonian admits a first integral that satisfies Paternain's hypotheses. Indeed, we will see that the first integral is \emph{nondegenerate in the Bott sense}. 
In section 4, prove that, for each torus  $\Tt\in\TT_\MM$, the geodesic flow $\phi$ satifies $\hp(\phi)=2$. Therefore,  such tori are cases of strict inequality for the theorem B.
\vspace{0.1cm}

  Burago et Ivanov showed that if  $g$ is a metric on $\T^n$, and if $B(x,r)$ is, as usual,  the ball with center $x$ and of radius $r$ in the universal cover $\R^n$, the limit $\Om(g)):=\lim_{r\rit+\infty} \frac{\Vol B(x,r)}{r^n}$ exists and is independent of $x$. It is the \emph{asymptotic volume of $g$}.
  They prove that $\Om(g)$ is equal, up to a constant $\ups$, to the volume $\VV$ of the unit ball of the \emph{stable norm} associated to $g$. 
In section 5, we will show that for $\Tt\in \TT_\MM$, the integrability of the geodesic Hamiltonian permits to compute \emph{explicitely} the volume $\VV_g$. Indeed, the stable norm coincide with \emph{Mather's $\beta$} function which can be explicitely determinated due to the existence of the first integral. 

 We recall the definitions of the stable norms and of Mather's function in paragraph 5.1. In paragraph 5.2, we determinate the expression of $\VV_g$ for tori in $\TT_\MM$. 

Finally observe that since the set  $\PP_\MM$ is dense in the Fréchet space $\PP$, our result is generic.

\subsubsection*{Aknowledgements.} The questions addressed in this paper were suggested to me by my advisor Jean-Pierre Marco. 
I want to thank him for his guidance and for many helpful discussions during the realization of this work. I am also grateful to Maxime Wolff for many instructive discussions about the growth of groups.

\section{Polynomial entropy}
In this short section, we briefly  define  the polynomial entropy. For a more complete introduction see \cite{Mar-09}.

Let $\Phi=(\phi_t)_{t\in\R}$ be a flow on a compact metric space $X$. For all $t\in\R$, we can construct as before the dynamical metrics associated to $\Phi$ by setting $d_t^\Phi(x,y)=\sup_{0\leq s<t}[d(\phi_s(x),\phi_s(y))]$.

We denote by $G_t^\phi(\eps)$ the minimal number of $d_t^\phi$-balls of radius $\eps$ in a covering of $X$

\begin{Def} The \emph{polynomial entropy} of $\Phi$ is defined by
\[
\hp(\phi)=\sup_{\eps>0}\limsup_{t\to\infty}\frac{\Log G_t^\phi(\eps)}{\Log t}=\sup_{\eps>0}\inf\left\{\sig\geq 0\,|\, \lim_{t\to\infty}\frac{1}{t^\sig}G_t^\phi(\eps)=0\right\}.
\]
\end{Def}

The  polynomial entropy $\hp$ is a $C^0$-conjugacy invariant and does not depend on the choice of topologically equivalent metrics on $X$.
 We emphasize the following  important fact.

\begin{rem}
When $\hto(f)>0$, $\hp(f)=+\infty$.
\end{rem}

Let us state the following resut proved in [LM].

\begin{prop}\label{Csub}  
Let $H:(\al,I)\ma h(I)$ be a $C^2$ Hamiltonian function on  $T^*\T^n$. Let $\Ss$ be a compact submanifold of $\R^n$, possibly with boundary. Then the compact $\T^n\times\Ss$ is invariant under the flow $\phi$ and one has
\[
\hp(\phi_{|_{\T^n\times\Ss}})=\max_{I\in \Ss} \rk \om(I),
\]
where $\om :\Ss\rit \R^n, I\mapsto d_I(h_{|_{\Ss}})$.
\end{prop}

\begin{cor}\label{entpolpourhconv}
If $h$ is strictly convex and if $\Ss$ is a compact regular energy level, one gets
$
\hp(\phi_{|_{\T^n\times\Ss}})=n-1.
$
\end{cor}

\section{Volume growth of balls and the  polynomial entropy.}
This section is  devoted to the proof of Theorem \ref{Manning}. Here $(M,g)$ is a compact Riemannian manifold with Riemannian cover $\wti{M}$.

\subsection{The number $\tau(M)$.}
Recall  that, for $x\in\wti{M}$:
\[
\tau(x)=\inf\{s\geq 0\,|\,\limsup_{r\rit\infty}\frac{1}{r^s}\Vol B(x,r)=0\}\leq \infty.
\]
We will first see that $\tau(x)$ is independent of $x$ and of $g$: it is a \emph{topological} invariant of $M$.

A \emph{growth function} is a nondecreasing function $\R^+\rit\R^+$. For any $x\in \tM$, the map $\nu_x:r\ma \Vol B(x,r)$ is a growth function.
To any nondecreasing function $\be:\N\rit\N$, one can associate the growth function $\al:t\rit \be(\lceil t\rceil)$. Therefore, the considerations below apply to nondecreasing functions $\N\rit \N$.

Two growth functions $\al_1$ and $\al_2$ are \emph{weakly equivalent} if there exist constants $\lam,\mu>1$ and $C,C'\geq 0$ such that for all $t\in\R^+$,
\[
\al_1(t)\leq \lam\al_2(\lam t+C) +C,\quad \al_2(t)\leq \mu\al_1(\mu t+C') +C'.
\]
We then write $\al_1\sim \al_2$.

Let $\Ga$ be a finitely generated group and let $S:=(s_1,\dots,s_p)$ be a set of generators. We denote by $\ell_S(\ga)$ the \emph{word length} of an element $\ga\in \Ga$, that is, the smallest integer $n$ for which there exists a sequence $(s_1,s_2,\dots, s_n)$ of elements of $S\cup S\inv$ such that $\ga=s_1s_2\dots s_n$. The \emph{word metric} $d_S$ on $\Ga$ is defined by
\[
d_S(\ga_1,\ga_2)=\ell_S(\ga_1\inv\ga_2).
\]
The group $\Ga$ endowed with $d_S$ acts on itself (by conjugacies) by isometries. The growth function of the pair $(\Ga,S)$ is defined by:
\[
 \be(\Ga,S;k):=\Card\{\ga\in \Ga\,|\, \ell_k(\ga)\leq k\}.
 \] 
 The \emph{exponential growth rate} of $(\Ga,S)$ is the upper limit :
 \[
 \om(\Ga, S):=\limsup_{k\rit\infty}\sqrt[k]{\be(\Ga,S,k)}.
 \]
Remark that if $S'$ is another finite set of generators of $\Ga$, there exists $\lam>0$ such that for all $k\in \N$, $\be(\Ga,S',k)\leq \be(\Ga, S,\lam k)$.  Therefore the group $\Ga$ is said to be of exponential growth if $\om(\Ga,S)>1$, of subexponetial growth if  $\om(\Ga,S)=1$ and of polynomial growth if there exists $d$ such that $\be(\Ga, S,k)\sim k^d$. These definitions make sense since these properties do not depend on the choice of $S$.

A \emph{quasi-isometry} between two metric spaces $(X,d)$ and $(X,d')$ is a map $f:X\rit X'$ such that there exist constants $\lam\geq 1$, $C\geq 0$ and $D\geq 0$ such that
\begin{itemize}
\item for any $z\in X'$ there exists $x\in X$ such that $d'(z,f(x))\leq D$
\item for all $(x,y)\in X^2$, $\frac{1}{\lam}d(x,y)-C\leq d'(f(x),f(y))\leq \lam d(x,y)+C$
\end{itemize} 

The next theorem is due to Milor and Schwarzc (see \cite{Ha-00} or \cite{BrHa-99}). 

\begin{thm}\label{Milnor-Scwarzc}
Let $(M,g)$ be a compact Riemannian manifold. We set $R:=\diam M:=\Max\{d(p,q)\,|\, (p,q)\in M^2\}$. Fix  $x\in \wti{M}$ and set $B:=B(x,R)$. Then

(i) The set $S:=\{\ga\in \pi_1(M)\,|\, \ga\neq 1\:\:{\rm{and}}\:\:\ga B\cap B\neq 0\}$ is a finite set of generators of $\pi_1(M)$.

(ii) The  number $r:= \Min\{d(B,\ga B)\,|\,\ga\in\pi_1(M),\,\ga \notin S\cap\{1\}\}$ is $>0$ and for all $\ga\in \pi_1(M)$, 
\[
\ell_S(\ga)\leq \frac{1}{r}d(x,\ga x)+1
\]

(iii) The map $\pi_1(M)\rit \wti{M}, \ga\ma \ga  x$ is a quasi-isometry.
\end{thm}

With the notation of the previous theorem, we denote by $\be$ the growth function of  the pair $(\pi_1(M),S)$.

\begin{cor}
For any $x\in \wti{M}$, the maps $\nu_x$ and $\be$ are weakly equivalent.
\end{cor}

\begin{proof}
We denote by $n_x$ the order of the isotropy subgroup of $x$ and we set $\lam:=\Max \{d(x,\ga x)\,\,x\in S\}$. 
Fix $k\in \N$. The closed balls $B(y, \frac{1}{3}r)$, for $y\in \pi_1(M)x$ are pairwise disjoint, so
\[
\frac{1}{n_x}\be(\pi_1(M),S;k)\,\nu_x\left(\frac{1}{3}r\right)\leq \nu_x\left(k\lam+\frac{1}{3}r\right).
\]
Conversely, let $y\in B(x,k)$. Since the set $\{\ga B\,|\, \ga \in \pi_1(M)\}$ is a covering of $\wti M$, there exists $\ga\in \pi_1(M)$, such that $y\in B(\ga x,R)$. Now by theorem \ref{Milnor-Scwarzc} (ii), 
\[
\ell_S(\ga)\leq  \frac{1}{r}d(x,\ga x)+1\leq \frac{1}{r}d(x,y)+\frac{R}{r}.\quad (*)
\]
Let $\de(x,y):=\frac{1}{r}d(x,y)+\frac{R}{r}$.
Then, the set $\{\ga B(x,R)\,|\, \ell_S(\ga)\leq \de(x,y)\}$ covers $B(x,k)$ and 
\[
\nu_x(k)\leq \be\left(\pi_1(M),S;k+\frac{R}{r}+1\right)\nu_x(R).\quad (**)
\]
Gathering $(*)$ and $(**)$, we conclude the proof.
\end{proof}

Therefore, if $\pi_1(M)$ has exponential or subexponential growth, $\nu(x)=+\infty$. If $\pi_1(M)$ has polynomial growth with degree $d$, $\nu(x)=d$. In particular, this immediately proves the following proposition.

\begin{prop}\label{tau_pourlestores} 
$\tau(\T^\ell)=\ell$.
\end{prop}

\subsection{Proof of Theorem \ref{Manning}.}
The Riemannian connexion on $M$ enables one to define a natural connexion on $TM$ in the following way. If $\nu:=(x,v)\in TM$, the parallel transports of $v$ along curves starting from $x$ give rise to curves $t\ma\ga(t)=(x(t),v(t))\in\,TM$. The \emph{horizontal subspace} $H(\nu)$ generated by the initial conditions $(\ga(0),\dot{\ga}(0))$ of these curves is complementary to the \emph{vertical subspace} $V(\nu):=\ker d_\nu\pi$. There exists a natural metric on $TM$, called the \emph{Sasaki metric} for which $H(\nu)$ and $V(\nu)$ are orthogonal and both isometric to $T_xM$.

In the following, $M$ is a \emph{compact} Riemannian manifold and $\wti{M}$ its universal cover (with the lifted metric).
Let us fix the notation.
\begin{itemize}
\item $\pi:TM\rit M$,\, $\ti{\pi}:T\wti{M}\rit \wti{M}$ and $p:\wti{M}\rit M$ are the canonical projections,
\item $d_M$ and $d_{\wti{M}}$ are the Riemannian distances on $M$ and $\tM$,
\item $d_{TM}$ and $d_{T\wti{M}}$ are the Riemannian distances on $TM$ and $T\tM$ associated with the Sasaki metric,
\item $(SM,d_{SM})$ and $(S\wti{M},d_{S\wti{M}})$ are the unit tangent bundles endowed with their induced metrics,
\item $\phi=(\phi_t)_{t\in\R}$ and $\ti{\phi}=(\ti{\phi}_t)_{t\in\R}$ are the geodesic flows on $SM$ and $S\tM$.
\end{itemize}
 
Recall  that for $\eps>0$, a subset $A$ of a compact metric space  $(X,d)$ is said to be \textit{$(t,\eps)$-separated} relatively to a continuous flow $(\psi_t)_{t\in\R}$ on $X$ if, for any $a$ and $b$ in $A$,  $\sup_{0\leq t'\leq t}d(\psi_t(a),\psi_t(b))\geq \eps$. 
If $S_t(\eps)$ is the maximal cardinal for a $(t,\eps)$-separated set, one has the following inequalities
\[
S_t(2\eps)\leq G_t(\eps)\leq S_t(\eps),
\]
where $G_t(\eps)$ is defined relatively to the distances induced by $\psi_t$.

\begin{proof}[Proof of Theorem \ref{Manning}]
It suffices to show that if $s> \hp(\phi)+1$, then $s>\tau(M)$. 
Fix $s>\hp(\phi)+1$ and let $\eta>0$. Then there exists $t_\eta>0$ such that for all $t\geq t_\eta$,
\[
\frac{1}{t^{s-1}}G_t\left(\frac{\eps}{2}\right)<\eta.
\]

Let $\rho$ such that for all $x\in \tM$, the projection $p:B(x,\rho)\rit M$ is injective. 
Fix $\eps>0$ such that $2\eps\leq \rho$. Let $t>0$. Let us construct a $(t,\eps)$-separated set in $SM$ for the flow $(\phi_t)$. Fix $x\in\wti{M}$ and let $C(x,t,t+\frac{\eps}{2})$ be the anular zone defined by
\[
C\left(x,t,t+\frac{\eps}{2}\right)=B\left(x,t+\frac{\eps}{2}\right)\setm B(x,t).
\]
Let $A$ be a $2\eps$-separated set in $C(x,t,t+\frac{\eps}{2})$ for the distance $d_{\wti{M}}$, that is, for any $(a,b)\in A^2$, $d_{\wti{M}}(a,b)\geq 2\eps$. For all $a\in A$, there exists a segment of unit speed geodesic $\ga_a$ with minimal length that joins $x$ and $a$. Necessarily, $\ell(\ga)\in[t,t+\frac{\eps}{2}]$.
We set $v_a=\dot{\ga}_a(0)\in S_x\tM$.
Then, for any $a$ and $b$ in $A$,
\begin{align*}
d_{S\wti{M}}\left(\ti{\phi}_t(v_a),\ti{\phi}_t(v_b)\right)& \geq d_{\wti{M}}\left(\ti{\pi}\circ\ti{\phi}_t(v_a),\ti{\pi}\circ\ti{\phi}_t(v_b)\right)\\
 & \geq d_{\wti{M}}(a,b)-d_{\wti{M}}\left(\ti{\pi}\circ\ti{\phi}_t(v_a),a\right)-d_{\wti{M}}\left(\ti{\pi}\circ\ti{\phi}_t(v_b),b\right)\\
 & \geq 2\eps-\frac{\eps}{2}-\frac{\eps}{2}=\eps.
\end{align*}
So $\{v_a\,|\,a\in A\}$ is $(t,\eps)$-separated.
Since the projection $T_xp$ is an isometry, $d_{SM}(d_xp(v_a),d_xp(v_b)=d_{S\wti{M}}(v_a,v_b)$, for all $(a,b)\in A^2$. So
\[
\sup_{0\leq t'\leq t}d_{SM}(\phi^t(v_a),\phi^t(v_b))\geq d_{SM}(d_xp(v_a),d_xp(v_b)\geq \eps
\]

Assume that $d_{S\wti{M}}(v_a,v_b)\geq\eps$. Then $\sup_{0\leq t'\leq t}d_{SM}(\phi^t(v_a),\phi^t(v_b))\geq \eps$ and the set $\Cal{A}=d_xp(\{v_A\,|\,a\in A\})$ is $(t,\eps)$-separated for the flow $\phi_t$.

 Assume that $d_{S\wti{M}}(v_a,v_b)\leq\eps$, then $d_{\wti{M}}(\ti{\pi}(v_a),\ti{\pi}(v_b))\leq\eps$.
Now by construction of $A$, $d_{\wti{M}}(\ti{\pi}\circ\ti{\phi}^t(v_a),\ti{\pi}\circ\ti{\phi}^t(v_b))\geq\eps$, so there exists $t_0\in\,[0,t]$ such that
\[
d_{\wti{M}}(\ti{\pi}\circ\ti{\phi}^{t_0}(v_a),\ti{\pi}\circ\ti{\phi}^{t_0}(v_b))=\eps.
\]
Therefore, since $2\eps\leq \rho$, one gets  $d_M(\pi\circ\phi_{t_0}(d_xp(v_a)),\pi\circ\phi_{t_0}(d_xp(v_b)))=\eps$
and the set $\Cal{A}$ is again $(t,\eps)$-separated for the flow $(\phi_t)$.

Now, notice that $\sup_{x\in\wti{M}}\Vol B(x,2\eps)$ is finite. Indeed, since $2\eps\leq \rho$, $\Vol B(x,2\eps)=\Vol B(p(x),2\eps)\leq \Vol M$. Let $\upsilon=\Vol M$.
For $t\geq t_\eta$,
\begin{multline}
\Vol C\left(x,t,t+\frac{\eps}{2}\right)\leq \upsilon \Card \Cal{A}\leq \upsilon\Card A\\
\leq \upsilon S_t(\eps)\leq \upsilon G_t\left(\frac{\eps}{2}\right)\leq \eta\upsilon t^{s-1}.
\end{multline}
Consequently, since $C(x,t,t+m\frac{\eps}{2})=\bigcup\limits_{k=0}^{m-1}C(x,t+k\frac{\eps}{2},t+(k+1)\frac{\eps}{2})$, one has, for $m\in \N^*$,
\[
\Vol C\left(x,t,t+m\frac{\eps}{2}\right)\leq \eta\upsilon\left(t^{s-1}+(t+\frac{\eps}{2})^{s-1}+\cdots+ (t+m\frac{\eps}{2})^{s-1}\right).
\]
Assume that $s\geq 1$. Then for each $k\in\N$, $(t+k\frac{\eps}{2})^{s-1}\leq \dfrac{2}{\eps}\displaystyle\int_{t+k\frac{\eps}{2}}^{t+(k+1)\frac{\eps}{2}}x^{s-1}dx$, which yields 
\[
\Vol C(x,t,t+m\frac{\eps}{2})\leq \eta\upsilon\frac{2}{\eps}\int_t^{t+m\frac{\eps}{2}}x^{s-1}dx\leq \eta\upsilon\frac{2}{s\eps}(t+m\frac{\eps}{2})^s.
\]
Assume that $0<s<1$. Then for each $k$,  $(t+k\frac{\eps}{2})^{s-1}\leq \dfrac{2}{\eps}\displaystyle\int_{t+(k-1)\frac{\eps}{2}}^{t+k\frac{\eps}{2}}x^{s-1}dx$, which yields 
\begin{multline*}
\Vol C(x,t,t+m\frac{\eps}{2})\leq \eta\upsilon\frac{2}{\eps}\int_{t-\frac{\eps}{2}}^{t+(m-1)\frac{\eps}{2}}x^{s-1}dx\\
\leq \eta\upsilon\frac{2}{\eps}\int_{t-\frac{\eps}{2}}^{t+m\frac{\eps}{2}}x^{s-1}dx\leq \eta\upsilon\frac{2}{s\eps}(t+m\frac{\eps}{2})^s.
\end{multline*}
In both cases, one gets, setting $\lam(\eps)=\upsilon\dfrac{2}{s\eps}$:
\begin{align*}
\Vol B(x,t+m\frac{\eps}{2}) & =\Vol B(x,t)+\Vol C(x,t,t+m\frac{\eps}{2})\\
& \leq \Vol B(x,t)+ \eta\lam(\eps)(t+m\frac{\eps}{2})^s.
\end{align*}
Finally:
\[
\limsup_{r\rit +\infty}\frac{\Vol B(x,r)}{r^s}\leq \sup_{t\in [0,\eps/2]}\limsup_{m\rightarrow\infty}\frac{\Vol B(x,t+m\frac{\eps}{2})}{(t+m\frac{\eps}{2})^s}\leq\eta\lam(\eps).
\]
Since $\eta$ is arbitrary, the limit above is zero and $s>\tau(M)$.
\end{proof}

%

\begin{exm}\textbf{The flat torus.}\: Let $g_b$ be the flat metric on $\T^\ell$ defined by a positive definite bilinear form $b$ on $\R^\ell$.  The cogeodesic flow $\phi_H$  on $S^*\T^\ell$ is in action-angle form, so by corollary \ref{entpolpourhconv} and proposition \ref{tau_pourlestores} one has
\[
\hp(\phi_H)=\ell-1=\tau(\T^\ell)-1
\] 
Since the geodesic flow $\phi$ and the cogeodesic flow $\phi_H$ are conjugate by the Legendre map, $\hp(\phi)=\ell-1$.
So the geodesic systems on the torus which are in action-angle form minimize  the  polynomial entropy. 
\end{exm}

\section{The polynomial entropy for tori $\TT_\MM$.} 

Let $M$ be a  $4$-dime\-nsional symplectic manifold and $H:M\rit \R$ be a smooth function. We denote by $\phi_H$ the Hamiltonian flow associated with $H$.
A first integral $f$ of $\phi_H$ is said to be \emph{nondegenerate in the Bott sense} on a compact  regular energy level $\EE$ of  $H$ if the critical  points of $f_{|\EE}$ form \emph{nondegenerate} strict submanifolds. By nondegenerate, we mean that $\partial^2f$ is nondegenerate on the normal subspaces to these submanifolds.
We say that the triple $(\EE,\phi_H,f)$ is a \emph{nondegenerate Bott system}. Such a system obviously satisfies the hypotheses of Paternain's theorem, and its topological entropy vanishes.

One proves that the critical submanifolds of $f_{|\EE}$ may only be circles (that is, periodic orbits of $\phi_H)$), tori or Klein bottles (see \cite{Mar-93}, \cite{F-88}). 
A nondegenerate Bott system is said to be \emph{dynamically coherent}  if the critical circles $\Cc$ are either elliptic periodic orbits or hyperbolic periodic orbits for $\phi_H$. The connected union of a hyperbolic orbit and its invariant manifold is called an ``eight-level'', we write $\8$-level.  In \cite{LM}, we  proved the following result.

\begin{thm}\label{hphw}  Let $(\EE,\phi_H,f)$ be a dynamically coherent system that possesses a hyperbolic orbit. Then
\[
\hp(\phi_H)=2.
\]
\end{thm}

We will now show that the geodesic flow (in restriction to every energy level) on a torus of $\TT_\MM$ is dynamically coherent and possesses  hyperbolic orbits. We begin with studying the critical set of $p_\vp$.
Fix $\Tt\in \TT$.
The Euclidean metric of $\R^3$ induces a Riemannian metric $g$ on $\Tt$.
The pullback $\ti{g}:=\Sig_{x,y}^* g$ of  $g$ is a Riemannian metric on $\R^2$ reads
\[
\forall\, (\vp,s)\in\R^2,\:\:\,
\ti{g}_{(\vp,s)}=
\begin{pmatrix}
4\pi^2x(s)^2 & 0\\
0 & r(s)^2
\end{pmatrix},
\]
where we denote by $r(s)$ the positive square root of $x'(s)^2+y'(s)^2$.
The Riemannian manifold $(\R^2,\ti{g})$ is the Riemannian cover of $(\Tt,g)$. 

The projection of $\ti{g}$ on $\T^2$ is also denoted by $g$, so  the quotient map $\ha{\Sig}_{x,y}$ becomes an isometry between  $(\T^2,g)$ and $(\Tt,g)$. 

\begin{nota} We set $\bm:=(\bar{\vp},\bar{s})\in\T^2$ and $m:=(\vp,s)\in\R^2$. If $\bm=\pi(m)$, the  spaces $T^*_{\bm}\T^2$ and $T^*_m\R^2$ are canonically  isometric. We denote by $p:=(p_\vp,p_s)$ their elements, so that the Liouville form on $T^*\T^2$ (resp. $T^*\R^2$) reads $\lam=p_\vp d\bvp+sd\bs$ (resp. $\lam=p_\vp d\vp+sds$). 
The functions $x$, $y$ and $r$ induce functions on $\T^2$, also denoted by $x$, $y$ and $r$.
\end{nota}
\vspace{0.2cm}

 The geodesic Hamiltonian $H$ on $T^*\R^2$ reads
\[
H(\vp,s,p_\vp,p_s)=\frac{1}{2}\left[\frac{p_{\varphi}^2}{4\pi^2x(s)^2}+\frac{p_s^2}{r(s)^2}\right].
\]
It is integrable in the Liouville sense, a first integral being the \emph{Clairaut integral} $p_\vp$.

It  projects in a natural way on a Hamiltonian function on $T^*\T^2$ also denoted by $H$. The associated Hamiltonian flows on $T^*\T^2$ and $T^*\R^2$ are respectively denoted by $(\phi_H^t)_{t\in\R}$ and $(\ti{\phi}_H^t)_{t\in\R}$.
Let $\vpi^* :T^*\R^2\rit T^*\T^2$, $\pi:T^*\T^2\rit \T^2$ and $\ti{\pi}:T^*\R^2\rit \R^2$ be the canonical projections. The following  diagram commutes.
\begin{equation*}
\xymatrix{
 T^*\R^2 \ar[r]^{\ti{\phi}_H^t} \ar[d]_{\vpi^*}  & T^*\R^2 \ar[r]^{\ti{\pi}} \ar[d]^{\vpi^*} & \R^2 \ar[d]^{\vpi} \\
    T^*\T^2 \ar[r]_{\phi_H^t} & T^*\T^2 \ar[r]_\pi & \T^2.
  }
  \end{equation*}

\begin{rem}\label{symetrie} 
If $\ha{m}$ stands for $m$ or $\bm$, the orbits of $(\ha{m},p)$ and $(\ha{m},-p)$ project by $\pi$ or $\tpi$ onto the same geodesic which they describe in opposite sense. We set $\zeta:(\ha{m},p)\ma (\ha{m},-p)$.
\end{rem}

\newcommand{\He}{H\inv(\{e\})}

Every regular energy level  $H^{-1}(\{e\})$ is a circle bundle  parametrized by 
\[
\T^3\ni(\bvp,\bs,\th)\mapsto (\bvp,\bs,2\pi\sqrt{2e}. x(s)\cos\theta,\sqrt{2e}. r(s)\sin\theta).
\]
Let $P_e$ be the restriction of $p_\vp$ to $\He$. We denote by $\Rr(e)$ the set of  regular values of $P_e$.

\begin{lem}\label{niveaudepphi} 
The set $\Rr(e)$ is a finite union of intervals $-I_k(e),  J(e)$ and $I_k(e)$ with $I_k(e)=\,]2\pi\sqrt{2e}x_k,2\pi\sqrt{2e}x_{k+1}[$ and $J(e)=\,]-2\pi\sqrt{2e}x_1,2\pi\sqrt{2e}x_1[$ where $0<x_1<x_2<\cdots<x_n$ are the critical values of $x$.
\end{lem}

\begin{proof}
Observe that $P_e(\bvp,\bs,\theta):=2\pi\sqrt{2e}.x(\bs)\cos\theta$ does not depend of $\bvp$. We set $\ha{P}_e:(\bs,\theta)\ma P(0,\bs,\theta)$.
 The critical points of $\ha{P}_e$ are the pairs $(\bs,\theta)$ such that:
$\theta =0\,[\pi]$ and $s$ is a critical point of $x$.
Since $x$ is a  Morse function, the set $\ov{S}$ of its critical points is finite and so is its set of  critical values. We set $x(\ov{S}):=\{x_1,\cdots,x_n\}$  with $x_1<x_2<\cdots< x_n$.
One just has to remark that $P_e\inv(\{\rho\})\neq \emptyset$ if and only if $\rho\in[-2\pi\sqrt{2e}x_n,2\pi\sqrt{2e}x_n]$.
\end{proof}

We denote by $\bs_i$ the unique critical point in $\T$ such that $x(\bs_i)=x_i$.

\begin{prop}
The system $(\He,\phi_H,P_e)$ is dynamically coherent and possesses a hyperbolic orbit.
\end{prop}

\begin{proof}
First we note that in the coordinates $(\bvp,\bs,\th)$, $X^H$ reads
\[
X^H(\bvp,\bs,\th)=\left(\frac{\sqrt{2e}\cos\th} {2\pi x(\bs)^2},\frac{\sqrt{2e}\sin\th}{r(\bs)},\frac{x'(\bs)}{r(\bs)x(\bs)}\right).
\]
By lemma \ref{niveaudepphi}, the critical loci of $P_e$ are the periodic orbits 
\[
\Cc_i^0:=\{(\bvp,\bs_i,0)\,|\,\bvp\in\T\}\,\quad {{\rm and}}\quad\Cc_i^\pi:=\{(\bvp,\bs_i,\pi)\,|\,\bvp\in\T\}
\]
 with period  $T_i:=\frac{\sqrt{2e}}{2\pi x_i^2}$. They are exchanged by the symetry $\zeta$, so we focus on the case where $\th=0$.  By simple computation, one checks that:\\
$\--$ if $x_i$ is a maximum, there exist $\al\in\,]0,\pi[$ such that the eigenvalues of $D\phi_{T_i}(\bvp,\bs_i,0)$ are $1,\,e^{i\al}$ and $e^{-i\al}$, and $\Cc_i^0$ is an elliptic orbit.\\
$\--$ if $x_j$ is a minimum,  there exists $\lambda>0$ such that the eigenvalues of $D\phi_{T_j}(\bvp,\bs_j,0)$ are $1,\,e^{\lambda }$ and $e^{-\lambda}$ and $\Cc_j^0$ is a hyperbolic orbit. 
\end{proof}

\begin{rem}\label{LagGraphPoly}
The level $P_e\inv(\{2\pi\sqrt{2e} x_1\})$ is the disjoint union of two $\8$-levels $\PP_e^0$ and $\PP_e^\pi$ exchanged by the symetry $\zeta$. 
Set $\th_e:\bs\ma \arccos\frac{x_1}{2\pi\sqrt{2e}x(\bs)}$.  
The complementary set  in $\PP_e^0$ of the circle $\Cc_1^0$ has the two following   connected components:
\[
 W_{e}^{0,+}:=\left\{(\bvp,\bs,\th_e(\bs)\,|\,(\bvp,\bs)\in \T\times\T\setm\{s_1\}\right\},
 \]
 and
\[ 
 W_{e}^{0,-}:=\left\{(\bvp,\bs,-\th_e(\bs))\,|\,(\bvp,\bs)\in \T\times\T\setm\{s_1\}\right\}.
 \]
The unions $\Gg_e^{0,+}$ of the orbit $\Cc_1^0$ and the submanifold $W_e^{0,+}$ is a Lagrangian Lipshitz graph over $\T^2$. We define in the same way the graph $\Gg_e^{0,-}$, and the graphs $\Gg_e^{\pi,+}$ and $\Gg_e^{\pi,-}$.
This particular property does not hold when when $x_i>x_1$.
\end{rem}

\section{The asymptotic volume for tori in $\TT\MM$.}

In this section, we briefly recall definitions of the stable norms for tori and state the result of Burago and Ivanov about the asymptotic volume.
 Then we briefly recall  the definitions of Mather's function $\al$ and $\be$.   
 For a more complete  introduction, we refer to \cite{Mas},  \cite{BBI} and \cite{P-99} for the stable norms and to \cite{Mat-91}, or to the very beautiful survey \cite{So-10} for Mather's theory.

\subsection{Stable norms and Mather's functions.}
Consider a Riemannian metric $g$ on the torus $\T^n$. With an element $\ga\in H_1(\T^n,\Z)$, we associate the set $\CC(\ga)$ of closed $C^1$ piecewise  curves  that represent $\ga$. We define a function $f$ on $H_1(\T^n,\Z)$ by setting:
\[
f(\ga):=\Inf\{\ell_g(c)\,|\, c\in \CC(\ga)\}.
\]
The function 
\[
\begin{array}{llll}
||\cdot||_s: & H_1(\T^n,\Z) & \rit &\R\\
& \ga &\ma & \lim\limits_{n\rit\infty}\frac{f(n\ga)}{n}
\end{array}
\]
extends to a norm on $H_1(\T^n,\R)$, called the \emph{stable norm associated with $g$}.

We denote by $\ti{g}$ the lifted metric on the universal cover $\R^n$ of $\T^n$ and by $\ups_g$ the Riemannian volume of $[0,1]^n$. Identifying $H_1(\T^n,\R)$ with $\R^n$ we denote by $\VV_g$ the volume of the unit ball of $||\cdot||_s$.

\begin{thm}\label{asymptoticvol} \textbf{(Burago-Ivanov)} For any $x$ in the universal cover $\R^n$ of $\T^n$, 
\[
\lim_{r\rit \infty}\frac{\Vol B(x,r)}{r^n}=\ups_g \VV(g).
 \]
\end{thm}

Consider a compact Riemannian manifold $M$ and fix a Tonelli Lagrangian $L$ on $TM$. We denote by $\phi_L$ its Euler-Lagrange flow and by $H$ its Fenchel-Legendre transform. The Hamiltonian flow associated with $H$ is denoted by $\phi_H$.

The orbits of $\phi_H$ are contained in the energy levels $H\inv({e})$ and those of $\phi_L$ in the subsets $\Ll\inv(H\inv({e}))$ (where $\Ll$ is the Legendre transform $TM\rit T^*M$). Due to the superlinearity, the sets $H\inv({e})$ and  $\Ll\inv(H\inv({e})$ are compacts.

 Let $\MM(L)$ be the set of probability measures $\mu$ on $TM$ that are invariant under $\phi_L$ and such that $\int_{TM}L d\mu<\infty$.
 The compactness of the sets $\Ll\inv(H\inv({e}))$ permits to prove that this set is non empty.
 The \emph{average action} $A_L$ on $\MM(L)$  defined by
\[
A_L(\mu)=\int_{TM}Ld\mu
\]
is lower semi-continuous and the the set of action-minimizing measures is non empty.

Given $\mu\in \MM(L)$, one  defines the following linear functional on $H^1(M,\R) $ by setting:
\[
\begin{matrix}
H^1(M,\R) & \rit &\R\\
[\eta] & \ma & \displaystyle\int_{TM}\eta d\mu,
\end{matrix}
\] 
 where $\eta$ is any representent of $[\eta]$. 
 The rotation vector, (or the homology class) of $\mu$ is the unique $\om(\mu)\in H_1(M,\R)$ such that
 \[
 \displaystyle\int_{TM}\eta d\mu= \langle\eta, \om(\mu)\rangle.
 \]
The map $\om :\MM(L)\rit H_1(M,\R)$ is continuous affine and surjective  and the \emph{Mather's function} $\beta$ is defined as
\[ 
\begin{matrix}
\be: H_1(M,\R)  & \rit & \R\\
\om &\ma & \min\limits_{\{\mu \in \MM(L)\,|\, \om(\mu)=\om\}}A_L(\mu)
\end{matrix}
\]
One checks that $\beta$ is a convex function.  
The following proposition is proved in \cite{Mas}.

\begin{prop} 
Assume that $L$ is a geodesic Lagrangian on $T\T^n$. Then $\be$ coincide with the stable norm $||\ga||_s$.
\end{prop}

If $\eta$ is a $1$-form on $TM$, it defines a new Tonelli Lagrangian on $TM$ by  setting:
\[
L_{\eta}(x,v)=L(x,v)-\langle\eta_x, v\rangle.
\]
If $\eta$ is closed, $L$ and $L_\eta$ have the same Euler-Lagrange flows.
Actually, changing the Lagrangian $L$ by a closed $1$-form does not perturb the dynamics. 
 Fix $c\in H^1(M,\R)$, and let $\eta_c$ be a representant of $\eta$. One says a measure $\mu \in \MM(L)$ is \emph{$c$-action-minimizing} if it minimizes $A_{L_{\eta_c}}$ among $\MM(L)$. 
 The \emph{Mather's function} $\al$ is defined as
\[ 
\begin{matrix}
\al: H^1(M,\R)  & \rit & \R\\
c &\ma & -\min\limits_{\{\mu \in \MM(L)}A_{L_{\eta_c}}(\mu)
\end{matrix}
\]
The function $\al$ and $\be$ are convex conjugate.

For $\om\in H_1(M,\R)$, we denote by $\MM^\om$ the subset of action-minimizing measures with rotation vector $\om$. For $c\in H^1(M,\R)$ we denote by $\MM_c$ the subset of $c$-action-minimizing measures.

\begin{Def}
The \emph{Mather set} $\wti\MM^\om$ \emph{of a rotation vector} $\om\in H_1(M,\R)$ is defined as:
\[
\wti{\MM}^\om:=\bigcup_{\mu \in \MM^\om} \supp \mu\subset TM
\]
The \emph{Mather set} $\wti\MM_c$ \emph{of cohomology class} $c\in H^1(M,\R)$ is defined as:
\[
\wti{\MM}_c:=\bigcup_{\mu \in \MM_c} \supp \mu\subset TM
\]
\end{Def}

\begin{thm}\textbf{Mather's graph theorem.}
The sets $\wti{\MM}_c$ and  $\wti{\MM}_L^\om$ are compact and $\phi_L$-invariant. If $\pi: TM\rit M$ is the canonical projection,  the restrictions $\pi_{|\wti{\MM}_c}$ and $\pi_{|\wti{\MM}^\om}$ are injective maps into $M$ whose inverses are Lipschitz. 
\end{thm}

If $H$ is a Tonelli Hamiltonian on $T^*M$, the Fenchel-Legendre inverse transform  defines a Tonelli Lagrangian $L_H$ on $TM$. Therefore, one can associate with $H$ the Mather's  functions defined by $L_H$. We denote them by $\beta_H$ and $\al_H$.

\begin{thm}
Let $H$  be a Tonelli Hamiltonian. Assume there exists an exact symplectomorphism $\Psi: T^*M\rit T^*M$ such that $H\circ \Psi$ is a Tonelli Hamiltonian. Then $\be_{H\circ\Psi}=\be_H$ and $\al_{H\circ \Psi}=\al_H$.
\end{thm}

\paragraph{Tonelli Hamiltonian on $T^*\T^n$ and admissible tori.} We now consider a Tonelli Hamiltonian on $T^*\T^n$. We denote by $L :T\T^n\rit \R$ its associated Lagrangian. As usual we denote by $\phi_H$ and $\phi_L$ their respective flows.

\begin{Def}
An \emph{admissible torus  with rotation vector} $\om$ is a torus $\Tt\subset T^*\T^n$ such that
\begin{enumerate}
\item $\Tt$ is a $C^1$ Lagrangian graph $\Tt:=\{(x, c + d_xu)\,|\, x\in \T^n\}$ whith $c\in \R^n$ and $u:\T^n\rit \R$,
\item $\Tt$ is $\phi_H$-invariant,
\item The restriction of $\phi_H$ to $\Tt$ is conjugate to the Kronecker flow $\phi^\om$ on $\T^n$ defined by $\phi_t^\om(x)=x+t\om$.
\end{enumerate}
\end{Def}
The following proposition is an easy consequence of the Fenchel-Legendre inequality.

\begin{prop}\label{KAM_minimise}
\begin{enumerate}
\item If $\ti{\mu}$ is another $\phi_L$-invariant probability measure with rotation vector $\om$, then $A_L(\mu)\leq A_L(\ti{\mu})$.
As a consequence $\Ll\inv(\Tt)=\wti{\MM}^\om$.
\item If $\ti{\mu}$ is another $\phi_L$-invariant probability measure, then $A_{L_c}(\mu)\leq A_{L_c}(\ti{\mu})$.
As a consequence $\Ll\inv(\Tt)=\wti{\MM}_c$.
\end{enumerate}
\end{prop}

Assume now that the Hamiltonian $H$ is in action-angle form, that is, $H(x,p)=h(p)$. Its associated Lagrangian is of the form $L(x,v)=\ell(v)$.
The cotangent bundle $T^*\T^n$ is globally foliated by admissible tori $\T\times\{c\}$ and the tangent bundle by $\phi_L$-invariant tori $\T\times\{\om\}$. 
Identifying $H^1(\T^n,\R)$ and $H_1(\T^n,\R)$ with $\R^n$, one easily deduces from  proposition \ref{KAM_minimise} that:
\[
\beta(\rho)=\ell(\rho),\quad {{\rm and}}\quad  \al(c)=h(c).
\]
In particular, one proves that 
\[
\wti{\MM}_c=\wti{\MM}^\om=\Ll\inv(\T^n\times\{c\})=\T^n\times\{\om\},
\]
when $\om=\nabla h(c)$ and $c=\nabla\ell(\om)$.

\subsection{The constant $\VV_g$ for tori of revolution}

We come back to the torus of revolution. We use the notation  of lemma \ref{niveaudepphi}.
For $e>0$, we set  
\begin{align*}
\Dd_e:= & P_e\inv(\{J(e)\})\\
\Zz_e^-:=  & P_e\inv([-2\pi\sqrt{2e}.x_n,-2\pi\sqrt{2e}.x_1[)\\ 
\Zz_e^+:=  & P_e\inv(]-2\pi\sqrt{2e}.x_1,-2\pi\sqrt{2e}.x_n]). 
\end{align*}

\subsubsection{The domains $\Dd_e$}For $e>0$ and $\rho\in J(e)$, we denote by $\th\lio$ the function  defined on $\T$ by $\th\lio: \bs\ma\arccos\frac{\rho}{2\pi\sqrt{2e}x(\bs)}$.
The set $\Dd_e$ has two connected components exchanged by the symetry $\zeta$ and foliated by Liouville tori:
\[
\Dd_e^+:= \bigcup_{\rho\in J(e)}\Tt_{e,\rho}^{+}\quad\textrm{and}\quad \Dd_e^-:= \bigcup_{\rho\in J(e)}\Tt_{e,\rho}^{-}
\]
where $\Tt\lio^+:=\{(\bvp,\bs,\th\lio(s))\,|\,(\bvp,\bs)\in \T^2\}$ and $\Tt\lio^-=\zeta(\Tt\lio^+)$.

The domains $\Dd_e^+$  and $\Dd_e^-$ are respectively  bounded by $\Gg_e^{0,+}$ and $\Gg_e^{\pi,+}$ and by $\Gg_e^{0,-}$ and $\Gg_e^{\pi,-}$. 

\vspace{-0.5cm}
\begin{figure}[h]
\begin{center}

\begin{pspicture}(4cm,5cm)
\rput(2,2.2){
\psset{xunit=1.cm,yunit=.45cm}
\parametricplot[linewidth=.02,plotpoints=200]{-180}{180}{t  90 div   t cos 4 add sqrt}
\parametricplot[linewidth=.02,plotpoints=200]{-180}{180}{t  90 div   t cos 4 add sqrt -1 mul}
\parametricplot[linewidth=.02,plotpoints=200]{-180}{180}{t  90 div   t cos 2 add sqrt}
\parametricplot[linewidth=.02,plotpoints=200]{-180}{180}{t  90 div   t cos 2 add sqrt -1 mul}
\parametricplot[linewidth=.03,plotpoints=200]{-180}{180}{t  90 div   t cos 1 add sqrt}
\parametricplot[linewidth=.03,plotpoints=200]{-180}{180}{t  90 div   t cos 1 add sqrt -1 mul}
\parametricplot[linewidth=.02,plotpoints=200]{-120}{120}{t  90 div   t cos .5 add sqrt}
\parametricplot[linewidth=.02,plotpoints=200]{-120}{120}{t  90 div   t cos .5 add sqrt -1 mul}
\pscircle[fillstyle=solid,fillcolor=black](0,0){.07}
\pscircle[fillstyle=solid,fillcolor=black](2,0){.07}
\pscircle[fillstyle=solid,fillcolor=black](-2,0){.07}
\rput(-2.2,3.6){$\th$}
\psline[linewidth=.05]{<-}(-.05,1.414)(.05,1.414)
\psline[linewidth=.05]{->}(-.05,-1.414)(.05,-1.414)
\psline[linewidth=.02]{->}(2.5,1.5)(1.1,0.9)
\psline[linewidth=.02]{->}(2.5,-1.5)(1.1,-0.9)
\rput(3,1.5){$\Gg_e^{\pi,-}$}
\rput(3,-1.5){$\Gg_e^{\pi,+}$}
}
\rput(2,0){
\psset{xunit=1cm,yunit=.45cm}
\psline[linewidth=.01]{->}(-2,-2.8)(-2,8.5)
\psline[linewidth=.01](2,-2.8)(2,8)
\psline[linewidth=.01]{->}(-2.5,0)(2.5,0)
\parametricplot[linewidth=.02,plotpoints=200]{-180}{180}{t  90 div   t cos 4 add sqrt}
\parametricplot[linewidth=.02,plotpoints=200]{-180}{180}{t  90 div   t cos 4 add sqrt -1 mul}
\parametricplot[linewidth=.02,plotpoints=200]{-180}{180}{t  90 div   t cos 2 add sqrt}
\parametricplot[linewidth=.02,plotpoints=200]{-180}{180}{t  90 div   t cos 2 add sqrt -1 mul}
\parametricplot[linewidth=.03,plotpoints=200]{-180}{180}{t  90 div   t cos 1 add sqrt}
\parametricplot[linewidth=.03,plotpoints=200]{-180}{180}{t  90 div   t cos 1 add sqrt -1 mul}
\parametricplot[linewidth=.02,plotpoints=200]{-120}{120}{t  90 div   t cos .5 add sqrt}
\parametricplot[linewidth=.02,plotpoints=200]{-120}{120}{t  90 div   t cos .5 add sqrt -1 mul}
\psline[linewidth=.02](-2,2.5)(2,2.5)
\psline[linewidth=.02](-2,-2.5)(2,-2.5)
\psline[linewidth=.02](-2,7.5)(2,7.5)
\pscircle[fillstyle=solid,fillcolor=black](0,0){.07}
\pscircle[fillstyle=solid,fillcolor=black](2,0){.07}
\pscircle[fillstyle=solid,fillcolor=black](-2,0){.07}
\rput(2.5,-.6){$\bar s$}
\psline[linewidth=.05]{->}(-.05,1.414)(.05,1.414)
\psline[linewidth=.05]{<-}(-.05,-1.414)(.05,-1.414)
\psline[linewidth=.02]{->}(2.5,1.5)(1.1,0.9)
\psline[linewidth=.02]{->}(2.5,-1.5)(1.1,-0.9)
\rput(3,1.5){$\Gg_e^{0,+}$}
\rput(3,-1.5){$\Gg_e^{0,-}$}

\rput(-2.3,7.5){$\frac{3\pi}{2}$}
\rput(-2.2,2.5){$\frac{\pi}{2}$}
\rput(-2.2,5){$\pi$}
\rput(-2.3,-2.5){$-\frac{\pi}{2}$}
\rput(-2.2,-.5){$\bar s_1$}
\rput(.3,.6){${\mathcal Z^+}$}
\rput(.3,5.6){${\mathcal Z^-}$}
\pscircle[fillstyle=solid,fillcolor=white](1.25,2.5){.35}
\rput(1.27,2.5){${\Dd}_e^+$}
\pscircle[fillstyle=solid,fillcolor=white](1.25,7){.35}
\rput(1.27,7){${\Dd}_e^-$}
\pscircle[fillstyle=solid,fillcolor=white](1.25,-2){.35}
\rput(1.27,-2){${\Dd}_e^-$}
}
\end{pspicture}
\end{center}
\vskip1.3cm
\caption{The domains $\Dd_e^+$ and $\Dd_e^-$.}
\end{figure}

\subsubsection{The domains $\Zz_e^\bullet$.} Since  $P_e\inv(\{\cup_{k=1}^{n-1}I_k(e)\})$ and $P_e\inv(\{\cup_{k=1}^{n-1}-I_k(e)\})$ are exchanged by $\zeta$, we focus on the first one. 
For $1\leq j\leq n$, we denote by $\PP_j^e$ the $\8$-level defined by $P_e(\PP_j^e)=x_j$.
It is the complementary set in $\Zz_e^+$ of the critical loci of $P_e$. So it has a finite number of connected components $D_{i,j}$, homotopic to $D^*\times\T$ where $D^*$ is the open pointed disc. Their boundary is either made of piece of a $\8$-level $\Pp_{j}^e$ and an elliptic orbit $\Ee_{i}$ with $x_i<x_j$ or two pieces of $\8$-levels  $\PP_{j}^e$ and $\PP_{i}^e$ with $x_i<x_j$.

More precisely, $D_{i,j}=\{(\bvp,\bs,\th\lio)\,|\,(\bvp,\bs)\in\T\times I_s,\, \rho\in\, ]2\pi\sqrt{2e}x_i,2\pi\sqrt{2e}x_j[\}$, where $\th\lio:s\ma \arccos\frac{\rho}{2\pi\sqrt{2e}x(s)}$ and where $I_s$ is a disjoint union of  intervals $]\bs_{i_1},\bs_{j_1}[$ and $]\bs_{j_2},\bs_{i_2}[$ with 
\begin{itemize}
\item $\bs_{i_1}< \bs_{i_2}$ in $x\inv(x_i)$  and $\bs_{j_1}\leq\bs_{j_2}$ in $x\inv(x_j)$,
\item $]\bs_{i_k},\bs_{j_k}[\,\cap\, \ov{S}=\emptyset$ and $]\bs_{i_l},\bs_{j_l}[\,\cap\, \ov{S}=\emptyset$.
\end{itemize}
\vspace{0.1cm}

The Liouville tori contained in $D_{i,j}$ are the connected union
\[
\Tt\lio:=\{(\bvp,\bs,\th\lio(\bs)),\,|\,(\bvp,\bs)\in \ov{B}\lio] \}\cup\{(\bvp,\bs,-\th\lio(\bs)),\,|\,(\bvp,\bs)\in \ov{B}\lio] \}.
\]
where $\ov{B}\lio:=\T\times[\ci\lio,\cs\lio]$ satisfies $[\bs_{j_1},\bs_{j_2}]\subset [\ci\lio,\cs\lio]\subset [\bs_{i_1},\bs_{i_2}]$

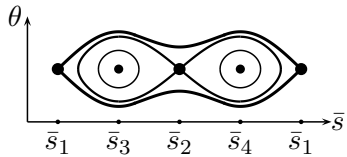
\begin{figure}[h]
\begin{center}
\begin{pspicture}(2cm,2.5cm)

\rput(1,2){
\psset{xunit=.4cm,yunit=.35cm}
\parametricplot[linewidth=.04,plotpoints=200]{-270}{-90}{t  60 div  1 add  t cos 1.7 add sqrt}
\parametricplot[linewidth=.04,plotpoints=200]{-270}{-90}{t  60 div  1 add  t cos 1.7 add sqrt -1 mul}
\parametricplot[linewidth=.04,plotpoints=200]{-540}{-310}{t  90 div   t cos 1 add sqrt}
\parametricplot[linewidth=.04,plotpoints=200]{-50}{180}{t  90 div   t cos 1 add sqrt}
\parametricplot[linewidth=.04,plotpoints=200]{-540}{-310}{t  90 div   t cos 1 add sqrt -1 mul}
\parametricplot[linewidth=.04,plotpoints=200]{-50}{180}{t  90 div   t cos 1 add sqrt -1 mul}
\parametricplot[linewidth=.03,plotpoints=200]{-360}{0}{t  90 div   t cos 1 add sqrt .87 mul}
\parametricplot[linewidth=.03,plotpoints=200]{0}{120}{t  90 div   t cos .5 add sqrt}
\parametricplot[linewidth=.03,plotpoints=200]{0}{120}{t  90 div   t cos .5 add sqrt -1 mul}
\parametricplot[linewidth=.03,plotpoints=200]{-360}{0}{t  90 div   t cos 1 add sqrt -.87 mul}
\parametricplot[linewidth=.03,plotpoints=200]{-360}{-480}{t  90 div   t cos .5 add sqrt}
\parametricplot[linewidth=.03,plotpoints=200]{-360}{-480}{t  90 div   t cos .5 add sqrt -1 mul}
\parametricplot[linewidth=.02,plotpoints=200]{-60}{60}{t  90 div   t cos -.5 add sqrt}
\parametricplot[linewidth=.02,plotpoints=200]{-60}{60}{t  90 div   t cos -.5 add sqrt -1 mul}
\parametricplot[linewidth=.02,plotpoints=200]{-300}{-420}{t  90 div   t cos -.5 add sqrt}
\parametricplot[linewidth=.02,plotpoints=200]{-300}{-420}{t  90 div   t cos -.5 add sqrt -1 mul}
\pscircle[fillstyle=solid,fillcolor=black](0,0){.06}
\pscircle[fillstyle=solid,fillcolor=black](-4,0){.06}
\pscircle[fillstyle=solid,fillcolor=black](2,0){.08}
\pscircle[fillstyle=solid,fillcolor=black](-2,0){.08}
\pscircle[fillstyle=solid,fillcolor=black](-6,0){.08}
\psline[linewidth=.02]{->}(-7,-2)(3,-2)
\psline[linewidth=.02]{->}(-7,-2)(-7,2)
\rput(3.2,-2){$\bs$}
\rput(-7.4,2){$\th$}
\pscircle[fillstyle=solid,fillcolor=black](-4,-2){.03}
\pscircle[fillstyle=solid,fillcolor=black](0,-2){.03}
\pscircle[fillstyle=solid,fillcolor=black](-2,-2){.03}
\pscircle[fillstyle=solid,fillcolor=black](-6,-2){.03}
\pscircle[fillstyle=solid,fillcolor=black](2,-2){.03}
\rput(-4,-2.8){$\bs_3$}
\rput(-2,-2.8){$\bs_2$}
\rput(0,-2.8){$\bs_4$}
\rput(-6,-2.8){$\bs_1$}
\rput(2,-2.8){$\bs_1$}
}

\end{pspicture}
\end{center}
\vskip-1.2cm
\caption{The domain $\Zz$ where there are 4 critical points}
\end{figure}

\subsection{The function $\beta$.}
By Mather's graph theorem, none of the levels of $p_\vp$ contained in the domains $\Zz_e^-$ or $\Zz_e^+$ can support a minimizing measure.

We set $\Dd^+_\infty:=\cup_{e>0}\Dd^+_e$ and $\tDd^+_\infty:=(\vpi^*)\inv(\Dd^+_\infty)$. We first remark that the Liouville tori $\Tt\lio$ contained in $\Dd^+_\infty$ and $\Dd^-_\infty$ are $C^1$ graphs over $\T^2$. Due to the symetry $\zeta$, we can focus on $\Dd_\infty^+$. 
It admits the following parametrization:
\[
\Dd^+_\infty=\{(\bvp,\bs,e,\rho)\in\T^2\times D\},
\] 
where $D:=\{(e,\rho)\,|\, e>0, \rho\in J(e)\}$.
By Arnol'd Liouville's Theorem  there exist an open domain $B_+\subset \R^2$ and a symplectic diffeomorphism
\[
\begin{matrix}
A_+ :  & \Dd^+_\infty & \rightarrow & \T^2\times B_+\\
 & (\bvp,\bs,e,\rho) & \mapsto & (\al^1,\al^2,I_1,I_2),
 \end{matrix}
\]
such that $I_1$, $I_2$  depend only on the value $(e,\rho)$ of the moment map $F:=(H,p_\vp)$ and generate $1$-periodic Hamiltonian flows. 
We denote by $H_+$ the Hamiltonian function on $\T^2\times B$ defined by $H_+(I)=H\circ A_+\inv(I)$.

In the same way, we set $H_-:=H\circ A_-\inv$, where $A_-$ is the action-angle transformation on $\Dd_\infty^-$.
\vspace{0.4cm}

\noindent\textbf{Consequence:} a) The flow $\phi_H$ on a torus $\Tt\lio$ is conjugate to a Kronecker flow on the torus $\T^2\times \{I(e,\rho)\}$, and the tori $\Tt\lio$ are admissible tori with rotation vector $\nabla H_+(I(e,\rho))$.
\vspace{0.2cm}

\noindent b) In the same way, the tori $\Tt\lio$ contained in $\Dd_\infty^-$ are admissible tori with rotation vector $\nabla H_-(I(e,\rho))$.
\vspace{0.4cm}

Finally, one proves that the graphs $\Gg_e^{0,+}$ and $\Gg_e^{\pi,+}$ and by $\Gg_e^{0,-}$ and $\Gg_e^{\pi,-}$ support a minimizing measure, indeed the support of such a measure is contained in the hyperbolic circle $\Cc_1^0$ and $\Cc_1^\pi$. 
\vspace{0.4cm}

\noindent \textbf{Consequence:} Roughly speaking, the function $\be$ associated with $H$ is the function $\be$ associated with the Hamiltonian $H$ in restriction to $\ov{\Dd_\infty^+\bigcup\Dd_\infty^-}$. 
\vspace{0.4cm}

Let us study the action-angle transformation $A_+$ and $A_-$. By symetry, we can focus on $A_+$. 
Set 
\[
\tau_{e,\rho}=\int_0^1\frac{r(t)}{\sqrt{2e-\frac{\rho^2} {4\pi^2x(t)^2}}}dt\:\:\:\textrm{and}\:\:\:\vp\lio:=\int_0^{\tau_{e,\rho}}\dot{\vp}(t)dt.
\]
In Appendix A, we prove that $A_+$ can be constructed such that:
\vspace{0.1cm}

$\bullet\, I_1(e,\rho)=\rho
\quad
\textrm{and}
\quad
I_2(e,\rho)=\displaystyle \int_0^1r(t)\sqrt{(2e-\frac{\rho^2}{4\pi^2x(t)^2})}dt.$
\vspace{0.2cm}

%

$\bullet\,\phi_H^t(m,p)  =\displaystyle\phi_{I_2}^{\frac{t}{\tau_{e,\rho}}}\circ\phi_{I_1}^{t\frac{\vp_{e,\rho}}{\tau_{e,\rho}}}(m,p)$.
\vspace{0.3cm}

\noindent Moreover, one checks that $A_+$ preserves the Liouville form.
The proof of the following proposition is given in Appendix B.

\begin{prop}
Let $h_+:B\rit \R$ be such that $H_+(\al, I)=h(I)$. Then $h$ is convex and superlinear.
\end{prop}

\begin{rem}\label{symetrie_action}
The action variables given by $A_-$  are $I_1^-(e,\rho)=I_1(e,\rho)$ and $I_2^-(e,\rho)=-I_2(e,\rho)$.
\end{rem}

\begin{cor}
The function $\be$ associated  with $H_{|_{\Dd_\infty^+}}$ coincide with the function $\be_+$  associated with the Hamiltonians $H_+$.
\end{cor}

Let $\om_+:B_+\rit \R^2: I\ma \nabla h_+(I)$. We can define in the same way $\om_-$. By remark \ref{symetrie_action}, one checks that $\om_-(I_-)$ is the image of $\om(I)$ by the map $(\om_1,\om_2)\ma (\om_1,-\om_2)$.
We set  $\Om^+=\om(h_+\inv(\{\demi\}))$ and $J:=J(\demi):=]-\rho_0,\rho_0[$ with $\rho_0:=2\pi x_1$.

\begin{prop}\label{valom}
The submanifold  $\Om^+$ is the image of the curve $\om$ parametrized by 
\[
\begin{matrix}
\om\;\; : & J & \longrightarrow & \R^2\\
 & \rho &\ma  & \om(\rho):= (X(\rho),Y(\rho))=\left(\frac{\vp_{\rho}}{\tau_{\rho}},\frac{1}{\tau_{\rho}}\right).
\end{matrix}
\]
\end{prop}

\begin{proof}
Let $(a,I)\in \R^2\times B_+$. Let $(m,p)$ such that $A(m,p)=(a,I)$. We set 
\[
(a^1(t),a^2(t)):=(a^1(\phi^t(m,p)),a^2(\phi^t(m,p)).
\] 
The result comes from:
$
\phi_{I_1}^{a^1(t)}\circ\phi_{I_2}^{a^2(t)}(\sig(e,\rho))=\phi_H^t(m,p)  =\phi_{I_2}^{\frac{t}{\tau_{e,\rho}}}\circ\phi_{I_1}^{t\frac{\vp_{e,\rho}}{\tau_{e,\rho}}}(m,p).
$
\end{proof}

\begin{lem}\label{equivalents} Let $\ga:=x''(0)$. We set $\tau_\rho:=\tau_{\demi,\rho}$,  $\vp_\rho:=\vp_{\demi,\rho}$ and $\tau'_\rho:=\dfrac{d\tau_\rho}{d\rho}(\rho)$. One has the following asymptotic estimates:
\begin{enumerate}
\item $\tau_\rho\simeq_{\rho\rit\rho_0}-\displaystyle\frac{\rho_0^{\frac{3}{2}}}{\rho} \frac{r(0)}{2\sqrt{\pi\ga}}\ln(\rho_0-\rho)$.
\item $\vp_\rho\simeq_{\rho\rit\rho_0}-\displaystyle\frac{1}{4\pi^2}\frac{1}{\sqrt{\rho_0}} \frac{r(0)}{2\sqrt{\pi\ga}}\ln(\rho_0-\rho)$.
\item $\tau'_\rho\simeq_{\rho\rit\rho_0}\displaystyle\frac{1}{4\pi^2}\frac{1}{\sqrt{\rho_0}} \frac{r(0)}{2\sqrt{\pi\ga}}\frac{1}{\rho_0-\rho}$.
\end{enumerate}
\end{lem}

\begin{cor}\label{prolongement_om}
 The curve $\om$ extends by continuity to the closed interval $\bar{J}$ with $\om(\rho_0):=(\frac{1}{4\pi^2}\frac{1}{\rho_0},0)$ and $\om(-\rho_0):=(-\frac{1}{4\pi^2}\frac{1}{\rho_0},0)$.
 We still denote by $\Om^+$ the image of $\om$.
\end{cor}

Recall that $T:=\frac{1}{4\pi^2}\frac{1}{\rho_0}$ is the period of the hyperbolic orbits $\Cc_1^{0}$ and $\Cc_1^\pi$.
The previous corollary may be interpreted in the following way.
\vspace{0.1cm}

\begin{rem}
The invariant measures $\mu$ associated with  the hyperbolic orbit $\Cc_1^{0}$ and $\Cc_1^\pi$ have respective rotation vectors $\om(\mu):=(T,0)$ and $\om(\mu):=(-T,0)$.
\end{rem}
\vspace{0.1cm}

Obviously, the previous corollary holds for the submanifold $\Om^-$.  
We set $\Om:=\Om^+(\bar{J})\cup \Om^-(\bar{J})$.

\begin{cor}\label{theorem B'}
The closed curve $\Om$ is the unit sphere of the stable norm associated with $g$. Therefore, the volume $\VV_g$ is the volume of the compact convex domain delimited by $\ov{\Om}$ and
\[
\VV_g=2\int_0^{\rho_0}X'(\rho)Y(\rho)d\rho=2\displaystyle\int_0^{\rho_0}\frac{\vp_\rho'\tau_\rho-\vp_\rho\tau_\rho'}{\tau_\rho^3}d\rho.
\]
\end{cor}

\appendix
\section{The action-angle variables for the torus of revolution} 
We focus on the domain $\Dd_\infty^+$ and we use Arnol'd's method ``by quadrature''.

For $T=(a,b)\in\R^2$ we denote by $\Phi^T$ the joint flow of the moment map $(H,p_\vp)$, that is,
$
\Phi^{(a,b)}(m,p):=\phi_H^a\circ\phi_{p_\vp}^b(m,p).
$

For any $\rho\in\Rr(e)$ the Liouville torus $\Tt_{e,\rho}$ is parametrized by $(\bvp,\bs)$ and a basis of $H_1(\Tt_{e,\rho},\Z)$ is given by $([\ga_1],[\ga_2])$ where
\[
\gamma_1(t)=(t,0)\,\,\,\, 
\textrm{and}\,\,\,\, 
\gamma_2(t)=(0,t).
\]
Let now $\bar{\sig}$ be any Lagrangian section of $F$ with equation $(\bvp,\bs)=(0,\bs_0)$. 
We set $\bar{\sig}_{e,\rho}=\bar{\sig}\cap\Tt_{e,\rho}$.
We look for a basis $(T_1,T_2)$ of the isotropy subgroup of $\Tt_{e,\rho}$ (that depends smoothly on $(e,\rho)$) such that:
\[
(t\mapsto\Phi^{tT_\ell}(\bar{\sig}_{e,\rho}))\in [\ga_\ell]\quad \ell=1,2,
\]
where $[\ga_\ell]$ denotes the rationnally homology class of the curve $\gamma_\ell$.
Denoting by $\lam$ the Liouville form on $T^*\T^2$, the action variable $I_1,I_2$ will be defined as 
\[
I_\ell:=\int_{\ga_\ell}\lambda.
\]

One checks that we can choose $T_1=(0,1)$ and $T_2=(\tau\lio,-\vp\lio)$ with
\[
\tau_{e,\rho}=\int_0^1\frac{r(t)}{\sqrt{2e-\frac{\rho^2} {4\pi^2x(t)^2}}}dt\:\:\:\textrm{and}\:\:\:\vp\lio:=\int_0^{\tau_{e,\rho}}\dot{\vp}(t)dt.
\]
This yields:
\[
I_1(e,\rho)=\rho
\quad
\textrm{and}
\quad
I_2(e,\rho)=\int_0^1r(t)\sqrt{(2e-\frac{\rho^2}{4\pi^2x(t)^2})}dt.
\]
From $\phi_{I_2}^t:=\phi_H^{t\tau\lio}\circ \phi_{\vp}^{-t\vp\lio}=\phi_H^{t\tau\lio}\circ \phi_{I_1}^{-t\vp\lio}$, one immediately deduces
\[
\phi_H^t=\phi_{I_2}^{\frac{t}{\tau\lio}}\circ \phi_{I_1}^{\frac{\vp\lio}{\tau\lio}}.
\]

\section{Convexity and superlinearity of $h$.}

Recall that the action variables are given by:
\[
I_1(e,\rho)=\rho
\quad
\textrm{and}
\quad
I_2(e,\rho)=\int_0^1r(t)\sqrt{(2e-\frac{\rho^2}{4\pi^2x(t)^2})}dt,
\]
where 
\[
(e,\rho)\in D:=\{(e,\rho)\,|\, e>0, \rho\in J(e)\}:=\{(e,\rho)\,|\, e>0, |\rho|\leq 2\pi\sqrt{2e}x_1\}.
\]
Let $f$ be the function  defined on $]\frac{1}{x(1)^2},+\infty[$ by :
\[
f : u\mapsto \int_0^1r(t)\sqrt{(u-\frac{1}{x(r)^2})}dt.
\]
It is an increasing bijection.  Denoting by $g$ its inverse, one has 
$
\frac{I_2}{I_1}=f\left(\frac{2h}{I_1^2}\right),
$
that is, 
\[
h(I_2,I_1)=\frac{I_1^2}{2}g\left(\frac{I_2}{I_1}\right).
\]

\subsection{Convexity of $h$.}
One has:
$
D^2h(I_1,I_2)=
\frac{1}{2}G$
where $G$ has the following form:
\[
G=
\begin{pmatrix}
g''\left(\frac{I_2}{I_1}\right) & g'\left(\frac{I_2}{I_1}\right)-\left(\frac{I_2}{I_1}\right)g''\left(\frac{I_2}{I_1}\right)\\
g'\left(\frac{I_2}{I_1}\right)-\left(\frac{I_2}{I_1}\right)g''\left(\frac{I_2}{I_1}\right) & 2g\left(\frac{I_2}{I_1}\right)-2\left(\frac{I_2}{I_1}\right)g'\left(\frac{I_2}{I_1}\right)+\left(\frac{I_2}{I_1}\right)^2g''\left(\frac{I_2}{I_1}\right)
\end{pmatrix}
\]
It suffices to show that the principal minors of this matrix are positive, that is 
$g''\left(\frac{I_2}{I_1}\right)>0$ and that $\det D^2\Hh(I_2,I_1)>0$.\\
Since $f$ is strictly concave and increasing, $g$ est strictly convex, thus $g''>0$. On the other hand,
\begin{align*}
\det D^2h(I_1,I_2) & =2g\left(\frac{I_2}{I_1}\right)g''\left(\frac{I_2}{I_1}\right)-g'^2\left(\frac{I_2}{I_1}\right)\\
& = \frac{1}{4}(g^2)''\left(\frac{I_2}{I_1}\right)+\frac{3}{2}g^2\left(\frac{I_2}{I_1}\right)(\log g)''\left(\frac{I_2}{I_1}\right)
\end{align*}
Since $g$ is convex, increasing and positive, $g^2$ is still convex, thus $(g^2)''>0$. Let us show that $\log g$ is convex. Since $\log g$ is an increasing bijection, it suffices to show that its inverse $\tilde{f}$ is concave.
We have
\[
\tilde{f}(u)=f(e^u)=\int_0^1\sqrt{r(t)(e^u-\frac{1}{x(r)^2})}dt,
\]
Then
\[
\tilde{f}'(u)=\int_0^1\frac{r(t)e^u}{\sqrt{r(t)(e^u-\frac{1}{x(r)^2})}}dt, 
\]
and
\[
\tilde{f}''(u)=\int_0^1\frac{-1}{x(t)^2}\frac{r(t)e^u}{(r(t)(e^u-\frac{1}{x(r)^2}))^{\frac{3}{2}}}dt < 0. 
\]

\subsection{Superlinearity of $h$.}
Set $k:=\max\{2\sqrt{2}x_1, \displaystyle\int_0^1r(t)dt\}$. Then 
\[
\Max(|I_1(e,\rho)|,|I_2(e,\rho)|)\leq k\sqrt{e},
\]
 that is, 
\[
\Max \{|I_1|,|I_2|\}\leq k\sqrt{h(I_1,I_2)},
\]
from which one immediately deduces the superlinearity.

\section{Asymptotic estimates for $\vp_\rho,\tau_\rho$ and $\tau'_\rho$}
Since $x'(0)=0$, one has
\[
\frac{1}{4\pi^2x(s)^2}\simeq_{s\rit 0}\frac{1}{\rho_0^2}\left(1-\frac{4\pi\ga}{\rho_0}s^2\right).
\]
Let $\al<\ga<\be$. There exists $\de_1>0$ such that forall $s\in]-\de,\de[$, 
\[
\begin{matrix}
\displaystyle\frac{1}{\rho_0^2}\left(1-\frac{4\pi\be}{\rho_0}s^2\right) &\leq &\displaystyle\frac{1}{4\pi^2x(s)^2} &\leq &\displaystyle\frac{1}{\rho_0^2}\left(1-\frac{4\pi\al }{\rho_0}s^2\right)\\
1-\displaystyle\frac{\rho^2}{\rho_0^2}\left(1-\frac{4\pi\al }{\rho_0}s^2\right) & \leq & 1-\displaystyle\frac{\rho^2}{4\pi^2x(s)^2} & \leq & 1-\displaystyle\frac{\rho^2}{\rho_0^2}\left(1-\frac{4\pi\be}{\rho_0}s^2\right)\\
\displaystyle\frac{1}{\sqrt{1-\frac{\rho^2}{\rho_0^2}\left(1-\frac{4\pi\be }{\rho_0}s^2\right)}} &\leq &\displaystyle\frac{1}{\sqrt{1-\frac{\rho^2}{4\pi^2x(s)^2} }}  &\leq & \displaystyle\frac{1}{\sqrt{1-\frac{\rho^2}{\rho_0^2}\left(1-\frac{4\pi\al }{\rho_0}s^2\right)}}\\
\displaystyle\frac{\rho_0}{\rho_0^2-\rho^2}\frac{1}{\sqrt{1+\frac{\rho^2}{\rho^2_0-\rho^2}\frac{4\pi\be }{\rho_0}s^2}} &\leq & \displaystyle\frac{1}{\sqrt{1-\frac{\rho^2}{4\pi^2x(s)^2} }} & \leq & \displaystyle\frac{\rho_0}{\rho_0^2-\rho^2}\frac{1}{\sqrt{1+\frac{\rho^2}{\rho^2_0-\rho^2}\frac{4\pi\al }{\rho_0}s^2}}
\end{matrix}
\]
1)$\tau_\rho=\displaystyle\int_{-\demi}^\demi\frac{r(s)ds}{\sqrt{1-\frac{\rho^2}{4\pi^2x(s)^2}}}$. Let $a<1<b$. There exists $\de_2$ such that for all $s\in]-\de_2,\de_2[$, $ar(0)\leq r(s)\leq br(0)$. Let $\de:=\min(\de_1,\de_2)$. For all $s\in\,]-\de,\de[$, one has:
\[
\begin{matrix}
\displaystyle\frac{\rho_0}{\rho_0^2-\rho^2}\frac{ar(0)}{\sqrt{1+\frac{\rho^2}{\rho^2_0-\rho^2}\frac{4\pi\be s^2}{\rho_0}}} &\leq & \displaystyle\frac{r(s)}{\sqrt{1-\frac{\rho^2}{4\pi^2x(s)^2} }} & \leq & \displaystyle\frac{\rho_0}{\rho_0^2-\rho^2}\frac{br(0)}{\sqrt{1+\frac{\rho^2}{\rho^2_0-\rho^2}\frac{4\pi\al s^2}{\rho_0}}}
\end{matrix}
\]
Hence

\begin{multline}
\displaystyle\frac{ar(0)\rho_0}{\rho_0^2-\rho^2}\int_{-\de}^\de\frac{ds}{\sqrt{1+\frac{\rho^2}{\rho^2_0-\rho^2}\frac{4\pi\be s^2}{\rho_0}}} \leq  \displaystyle\int_{-\de}^\de\frac{r(s)ds}{\sqrt{1-\frac{\rho^2}{4\pi^2x(s)^2} }}\\
 \leq  \displaystyle\frac{br(0)\rho_0}{\rho_0^2-\rho^2}\int_{-\de}^\de\frac{ds}{\sqrt{1+\frac{\rho^2}{\rho^2_0-\rho^2}\frac{4\pi\al s^2}{\rho_0}}}.
\end{multline}

Let $\zeta:=\frac{\sqrt{\rho_0^2-\rho^2}}{\rho}$, let $k$ stands for $4\pi\be$ or $4\pi\al$ and $c$ stands for $a$ or $b$. One has using the change of variable $u=\sqrt{\frac{k}{\rho_0}}\frac{1}{\zeta}s$:
\begin{align*}
\frac{c\rho_0r(0)}{\rho\zeta}\int_{-\de}^\de \frac{ds}{\sqrt{1+\frac{k}{\rho_0}\frac{s^2}{\zeta^2}}} 
 &=\frac{c\rho_0r(0)}{\rho\zeta}\frac{\zeta\sqrt{\rho_0}}{\sqrt{{k}}}\int\limits_{-\sqrt{\frac{k}{\rho_0}}\frac{\de}{\zeta}}^{\sqrt{\frac{k}{\rho_0}}\frac{\de}{\zeta}}\frac{du}{\sqrt{1+u^2}}\\
&=2\frac{\rho_0^{\frac{3}{2}}}{\rho} \frac{cr(0)}{\sqrt {k}}\argsh\left(\sqrt{\frac{k}{\rho_0}}\frac{\de}{\zeta}\right)\\
& =-\demi\ln(\rho_0-\rho)+\ln\left(\frac{\rho\sqrt{\frac{k}{\rho_0}}\de}{\sqrt{\rho_0+\rho}}\left(1+\sqrt{1+\frac{\rho_0^2-\rho^2}{\rho^2\frac{k}{\rho_0}\de^2}}\right)\right).
\end{align*}
If $f(\rho)=\ln\left(\frac{\rho\sqrt{\frac{k}{\rho_0}}\de}{\sqrt{\rho_0+\rho}}\left(1+\sqrt{1+\frac{\rho_0^2-\rho^2}{\rho^2\frac{k}{\rho_0}\de^2}}\right)\right)$, $f$ is bounded over $[0,\rho_0]$. Then:
\[
\frac{\rho_0^{\frac{3}{2}}}{\rho} \frac{ar(0)}{2\sqrt{\pi\be}}\leq\liminf_{\rho\rit\rho_0}\frac{1}{\ln(\rho_0-\rho)}\leq\limsup_{\rho\rit\rho_0}\frac{1}{\ln(\rho_0-\rho)}\leq \frac{\rho_0^{\frac{3}{2}}}{\rho} \frac{br(0)}{2\sqrt{\pi\al}}.
\]
Since these inequalities holds for any $a<1<b$ and any $\al<\ga<\be$, one has 
\[
\displaystyle\int_{-\de}^\de \frac{r(s)ds}{\sqrt{1-\frac{\rho^2}{4\pi^2x(s)^2}}}\simeq_{\rho\rit\rho_0}-\frac{\rho_0^{\frac{3}{2}}}{\rho} \frac{r(0)}{2\sqrt{\pi\ga}}.
\]
Now since $\Tt_\rho=\displaystyle\int_{-\de}^\de\frac{r(s)ds}{\sqrt{1-\frac{\rho^2}{4\pi^2x(s)^2}}}+\displaystyle\int_\de^\demi\frac{r(s)ds}{\sqrt{1-\frac{\rho^2}{4\pi^2x(s)^2}}}+\displaystyle\int_{-\demi}^{-\de}\frac{r(s)ds}{\sqrt{1-\frac{\rho^2}{4\pi^2x(s)^2}}}$ and since the two last integrals are uniformly bounded on $[0, \rho_0]$, one gets the first equivalent.

2)$\vp_\rho=\displaystyle\int_{-\demi}^\demi\frac{\rho}{x(s)^2}\frac{r(s)ds}{\sqrt{1-\frac{\rho^2}{4\pi^2x(s)^2}}}$. Let $a<1<b$. There exists $\de'_2$ such that for all $s\in]-\de_2',\de_2'[$, 
\[
\frac{ar(0)}{4\pi^2\rho_0^2}\leq \frac{r(s)}{x(s)^2}\leq \frac{br(0)}{4\pi^2\rho_0^2}. 
\]
Let $\de':=\min(\de_1,\de'_2)$. One has:
\begin{multline}
\displaystyle\frac{\rho}{4\pi^2\rho_0^2}\frac{ar(0)\rho_0}{\rho_0^2-\rho^2}\int_{-\de}^\de\frac{ds}{\sqrt{1+\frac{\rho^2}{\rho^2_0-\rho^2}\frac{4\pi\be s^2}{\rho_0}}}\\
\leq  \displaystyle\int_{-\de'}^{\de'}\frac{\rho}{x(s)^2}\frac{r(s)ds}{\sqrt{1-\frac{\rho^2}{4\pi^2x(s)^2} }}\\
\leq \displaystyle\frac{\rho}{4\pi^2\rho_0^2}\frac{br(0)\rho_0}{\rho_0^2-\rho^2}\int_{-\de}^\de\frac{ds}{\sqrt{1+\frac{\rho^2}{\rho^2_0-\rho^2}\frac{4\pi\al s^2}{\rho_0}}}.
\end{multline}
The end of the calculus is similar to the previous one and one gets the second equivalent.

3)$\tau'_\rho=\displaystyle\int_{-\demi}^\demi\frac{\rho}{4\pi x(s)^2}\frac{r(s)ds}{(1-\frac{\rho^2}{4\pi^2x(s)^2})^{\frac{3}{2}}}$. Let $a<1<b$. There exists $\de''_2$ such that for all $s\in]-\de_2'',\de_2''[$, 
\[
\frac{ar(0)}{16\pi^4\rho_0^2}\leq \frac{r(s)}{4\pi^2x(s)^2}\leq \frac{br(0)}{16\pi^4\rho_0^2}. 
\]
Let $\de'':=\min(\de_1,\de''_2)$. One has:
\begin{multline}
\displaystyle\frac{ar(0)}{16\pi^4}\frac{\rho\rho_0}{(\rho_0^2-\rho^2)^{\frac{3}{2}}}\int_{-\de''}^{\de''}\frac{ds}{(1+\frac{\rho^2}{\rho^2_0-\rho^2}\frac{4\pi\be s^2}{\rho_0})^{\frac{3}{2}}}\\
 \leq  \displaystyle\int_{-\de''}^{\de''}\frac{\rho}{4\pi^2 x(s)^2}\frac{r(s)ds}{(1-\frac{\rho^2}{4\pi^2x(s)^2})^{\frac{3}{2}} }\\
\leq \displaystyle\frac{br(0)}{16\pi^4}\frac{\rho\rho_0^2}{(\rho_0^2-\rho^2){\frac{3}{2}}}\int_{-\de''}^{\de''}\frac{ds}{(1+\frac{\rho^2}{\rho^2_0-\rho^2}\frac{4\pi\al s^2}{\rho_0})^{\frac{3}{2}}} .
\end{multline}
Using the same change of variables $u=\sqrt{\frac{k}{\rho_0}}\frac{1}{\zeta}s$, one gets if $k$ stands for $4\pi\be$ or $4\pi\al$: 
\begin{align*}
\int_{-\de''}^{\de''}\frac{ds}{(1+\frac{\rho^2}{\rho^2_0-\rho^2}\frac{4\pi\be s^2}{\rho})^{\frac{3}{2}}} &
=\zeta\sqrt{\frac{\rho_0}{k}}\int\limits_{-\sqrt{\frac{k}{\rho_0}}\frac{\de}{\zeta}}^{\sqrt{\frac{k}{\rho_0}}\frac{\de}{\zeta}}\frac{du}{(1+u^2)^{\frac{3}{2}} }\\
& =2\zeta\sqrt{\frac{\rho_0}{k}}\sqrt{\frac{k}{\rho_0}}\frac{\de}{\zeta}\left(\sqrt{1+\frac{k}{\rho_0}\frac{\de''^2}{\zeta^2}}\right)\inv\\
& = 2\zeta\sqrt{\frac{\rho_0}{k}}\left(\sqrt{1+\frac{\rho_0}{k}\frac{\zeta^2}{\de''^2}}\right)\inv\\
& =\frac{\sqrt{\rho_0^2-\rho^2}}{\rho}\frac{\sqrt{\rho_0}}{k}\left(\sqrt{1+\frac{\rho_0}{k}\frac{\zeta^2}{\de''^2}}\right)\inv.
\end{align*}

Hence if $c$ stands for $a$ or $b$:
\begin{multline*}
\frac{cr(0)}{16\pi^4}\frac{\rho\rho_0}{(\rho_0^2-\rho^2)^{\frac{3}{2}}}\int_{-\de''}^{\de''}\frac{ds}{(1+\frac{\rho^2}{\rho^2_0-\rho^2}\frac{4\pi\be s^2}{\rho})^{\frac{3}{2}}}
=\frac{cr(0)}{16\pi^4}\frac{1}{\sqrt{k\rho_0}}\frac{\rho_0^2}{\rho_0^2-\rho^2}\left(\sqrt{1+\frac{\rho_0}{k}\frac{\zeta^2}{\de''^2}}\right)\inv.
\end{multline*}
Now  since $\frac{\rho_0^2}{\rho_0^2-\rho^2}=\demi\left(\frac{\rho_0}{\rho_0-\rho}+\frac{\rho_0}{\rho_0+\rho}\right)$, one has the following equivalent:
\[
\frac{cr(0)}{16\pi^4}\frac{\rho\rho_0}{(\rho_0^2-\rho^2)^{\frac{3}{2}}}\int_{-\de''}^{\de''}\frac{ds}{(1+\frac{\rho^2}{\rho^2_0-\rho^2}\frac{4\pi\be s^2}{\rho})^{\frac{3}{2}}}\simeq_{\rho_0\rit\rho}\frac{cr(0)}{16\pi^4}\frac{\sqrt{\rho_0}}{\sqrt{k}}\frac{1}{\rho_0-\rho}.
\]
As before, since $a<1<b$ and $\al<\ga<\be$ are arbitrary, one gets 
\[
\int_{-\de''}^{\de''}\frac{\rho}{4\pi^2 x(s)^2}\frac{r(s)ds}{(1-\frac{\rho^2}{4\pi^2x(s)^2})^{\frac{3}{2}} }\simeq_{\rho_0\rit\rho}\frac{r(0)}{32\pi^4}\frac{\sqrt{\rho_0}}{\sqrt{\pi\ga}}\frac{1}{\rho_0-\rho}.
\]
To conclude, one juste has to observe that the two integrals 
\[
\int_{-\demi}^{-\de''}\frac{\rho}{4\pi x(s)^2}\frac{r(s)ds}{(1-\frac{\rho^2}{4\pi^2x(s)^2})^{\frac{3}{2}}}\:\:\textrm{and}\:\:\int_{\de''}^\demi\frac{\rho}{4\pi x(s)^2}\frac{r(s)ds}{(1-\frac{\rho^2}{4\pi^2x(s)^2})^{\frac{3}{2}}}
\] are bounded on $[0,\rho_0]$.

\bibliographystyle{alpha}
\bibliography{biblioplygroth}

\begin{thebibliography}{Fom88}

\bibitem[BBI01]{BBI}
Dmitri Burago, Yuri Burago, and Sergei Ivanov.
\newblock {\em A course in metric geometry}, volume~33 of {\em Graduate Studies
  in Mathematics}.
\newblock American Mathematical Society, Providence, RI, 2001.

\bibitem[BH99]{BrHa-99}
M.~R. Bridson and A.~Haefliger.
\newblock {\em Metric spaces of non-positive curvature}, volume 319 of {\em
  Grundlehren der Mathematischen Wissenschaften [Fundamental Principles of
  Mathematical Sciences]}.
\newblock Springer-Verlag, Berlin, 1999.

\bibitem[dlH00]{Ha-00}
Pierre de~la Harpe.
\newblock {\em Topics in geometric group theory}.
\newblock Chicago Lectures in Mathematics. University of Chicago Press,
  Chicago, IL, 2000.

\bibitem[Fom88]{F-88}
A.~T. Fomenko.
\newblock {\em Integrability and nonintegrability in geometry and mechanics},
  volume~31 of {\em Mathematics and its Applications (Soviet Series)}.
\newblock Kluwer Academic Publishers Group, Dordrecht, 1988.
\newblock Translated from the Russian by M. V. Tsaplina.

\bibitem[LM]{LM}
C.~Labrousse and J-P. Marco.
\newblock Polynomial entropies for bott non degenerate hamiltonian systems.

\bibitem[Man79]{M-79}
A~Manning.
\newblock Topological entropy for geodesic flows.
\newblock {\em Ann. of Math.}, 110:567--573, 1979.

\bibitem[Mar93]{Mar-93}
J.-P. Marco.
\newblock Obstructions topologiques à l'intégrabilité des flots géodésiques en
  classe de bott.
\newblock {\em Bull. Sc math.}, 117:185--209, 1993.

\bibitem[Mar09]{Mar-09}
J-P Marco.
\newblock Dynamical complexity and symplectic integrability.
\newblock ArXiv, 2009.

\bibitem[Mas96]{Mas}
D.~Massart.
\newblock {\em Normes stables des surfaces}.
\newblock TEL, 1996.
\newblock Thesis (Ph.D.)--Ecole Normale Supérieure de Lyon.

\bibitem[Mat91]{Mat-91}
John~N. Mather.
\newblock Action minimizing invariant measures for positive definite
  {L}agrangian systems.
\newblock {\em Math. Z.}, 207(2):169--207, 1991.

\bibitem[Pat91]{Pat-dim4}
Gabriel Paternain.
\newblock Entropy and completely integrable {H}amiltonian systems.
\newblock {\em Proc. Amer. Math. Soc.}, 113(3):871--873, 1991.

\bibitem[Pat99]{P-99}
G~Paternain.
\newblock {\em Geodesic flows}.
\newblock Birkaüser, 1999.

\bibitem[Sor10]{So-10}
A.~Sorrentino.
\newblock Lecture notes on {M}ather's theory for {L}agrangian systems.
\newblock ArXiv:1011.0590v1, 2010.

\end{thebibliography}

\end{document}